\documentclass[10pt,draft]{article}

\AtEndDocument{\bigskip{\footnotesize%
  \textsc{University of Luxembourg, Maison du Nombre, 6 avenue de la Fonte, L-4364 Esch-sur-Alzette, Grand-Duchy of Luxembourg.} \par
  \textit{E-mail address:} \texttt{francois.petit@uni.lu} 
  }}

\usepackage[T1]{fontenc}
\usepackage[latin1]{inputenc}
\usepackage{amsmath,amssymb,amsthm,mathrsfs}
\usepackage[all,2cell]{xy}
\usepackage{graphicx}
\usepackage{enumerate} 
\usepackage{lmodern}
\usepackage{mathrsfs} 
\usepackage{dsfont}
\usepackage{color}

\DeclareMathOperator{\Aut}{Aut}

\DeclareMathOperator{\id}{id}
\DeclareMathOperator{\opp}{op}
\DeclareMathOperator{\Mod}{Mod}

\DeclareMathOperator{\Hn}{H}

\DeclareMathOperator{\End}{End}

\DeclareMathOperator{\codim}{codim}

\begin{document}

\theoremstyle{plain} 
\newtheorem{theorem}{Theorem}[section]
\newtheorem{corollary}[theorem]{Corollary}
\newtheorem{proposition}[theorem]{Proposition}
\newtheorem{lemma}[theorem]{Lemma}
\theoremstyle{definition} 
\newtheorem{definition}[theorem]{Definition}
\newtheorem{example}[theorem]{Example}
\newtheorem{examples}[theorem]{Examples}
\newtheorem{remark}[theorem]{Remark}
\newtheorem{notation}[theorem]{Notations}

\numberwithin{equation}{section}

\newcommand{\pinv}{p_\ast^{\C^\times}\hspace{-0.2em}}
\newcommand{\On}[1]{\mathcal{O}_{#1}}
\newcommand{\En}[1]{\mathcal{E}_{#1}}
\newcommand{\Fn}[1]{\mathcal{F}_{#1}} 
\newcommand{\tFn}[1]{\mathcal{\tilde{F}}_{#1}}
\newcommand{\hum}[1]{hom_{\mathcal{A}}({#1})}
\newcommand{\hcl}[2]{#1_0 \lbrack #1_1|#1_2|\ldots|#1_{#2} \rbrack}
\newcommand{\hclp}[3]{#1_0 \lbrack #1_1|#1_2|\ldots|#3|\ldots|#1_{#2} \rbrack}
\newcommand{\partiel}[2]{\dfrac{\partial #1}{\partial #2}}
\newcommand{\catMod}{\mathsf{Mod}}
\newcommand{\Der}{\mathsf{D}}
\newcommand{\Ds}{D_{\mathbb{C}}}
\newcommand{\DG}{\mathsf{D}^{b}_{dg,\mathbb{R}-\mathsf{C}}(\mathbb{C}_X)}
\newcommand{\lI}{[\mspace{-1.5 mu} [}
\newcommand{\rI}{] \mspace{-1.5 mu} ]}
\newcommand{\Ku}[2]{\mathfrak{K}_{#1,#2}}
\newcommand{\iKu}[2]{\mathfrak{K^{-1}}_{#1,#2}}
\newcommand{\Be}{B^{e}}
\newcommand{\op}[1]{#1^{\opp}}
\newcommand{\N}{\mathbb{N}}
\newcommand{\Ab}[1]{#1/\lbrack #1 , #1 \rbrack}
\newcommand{\Du}{\mathbb{D}}
\newcommand{\C}{\mathbb{C}}
\newcommand{\Z}{\mathbb{Z}}
\newcommand{\w}{\omega}
\newcommand{\K}{\mathcal{K}}
\newcommand{\Hoc}{\mathcal{H}\mathcal{H}}
\newcommand{\env}[1]{{\vphantom{#1}}^{e}{#1}}
\newcommand{\eA}{{}^eA}
\newcommand{\eB}{{}^eB}
\newcommand{\eC}{{}^eC}
\newcommand{\cA}{\mathcal{A}} 
\newcommand{\cB}{\mathcal{B}}
\newcommand{\cD}{\mathcal{D}}
\newcommand{\cR}{\mathcal{R}}
\newcommand{\cI}{\mathcal{I}}
\newcommand{\cL}{\mathcal{L}}
\newcommand{\cO}{\mathcal{O}}
\newcommand{\cM}{\mathcal{M}}
\newcommand{\cN}{\mathcal{N}}
\newcommand{\cK}{\mathcal{K}}
\newcommand{\cC}{\mathcal{C}}
\newcommand{\cF}{\mathcal{F}}
\newcommand{\cG}{\mathcal{G}}
\newcommand{\cP}{\mathcal{P}}
\newcommand{\cQ}{\mathcal{Q}}
\newcommand{\cU}{\mathcal{U}}
\newcommand{\cE}{\mathcal{E}}
\newcommand{\cS}{\mathcal{S}}
\newcommand{\cT}{\mathcal{T}}
\newcommand{\chE}{\widehat{\mathcal{E}}}
\newcommand{\cW}{\mathcal{W}}
\newcommand{\chW}{\widehat{\mathcal{W}}}
\newcommand{\Hper}{\Hn^0_{\textrm{per}}}
\newcommand{\Dper}{\Der_{\mathrm{perf}}}
\newcommand{\Yo}{\textrm{Y}}
\newcommand{\gqcoh}{\mathrm{gqcoh}}
\newcommand{\coh}{\mathrm{coh}}
\newcommand{\cc}{\mathrm{cc}}
\newcommand{\qcc}{\mathrm{qcc}}
\newcommand{\gd}{\mathrm{gd}}
\newcommand{\qcoh}{\mathrm{qcoh}}
\newcommand{\lcl}{\mathrm{lcl}}
\newcommand{\fin}{\mathrm{fin}}
\newcommand{\obplus}[1][i \in I]{\underset{#1}{\overline{\bigoplus}}}
\newcommand{\Lte}{\mathop{\otimes}\limits^{\rm L}}
\newcommand{\te}{\mathop{\otimes}\limits^{}}
\newcommand{\btimes}{\mathop{\boxtimes}\limits^{}}
\newcommand{\pt}{\textnormal{pt}}
\newcommand{\A}[1][X]{\cA_{{#1}}}
\newcommand{\dA}[1][X]{\cC_{X_{#1}}}
\newcommand{\conv}[1][]{\mathop{\circ}\limits_{#1}}
\newcommand{\sconv}[1][]{\mathop{\ast}\limits_{#1}}
\newcommand{\reim}[1]{\textnormal{R}{#1}_!}
\newcommand{\roim}[1]{\textnormal{R}{#1}_\ast}
\newcommand{\ldetens}{\overset{\mathnormal{L}}{\underline{\boxtimes}}}
\newcommand{\br}{\bigr)}
\newcommand{\bl}{\bigl(}
\newcommand{\sC}{\mathscr{C}}
\newcommand{\ucat}{\mathbf{1}}
\newcommand{\ubtimes}{\underline{\boxtimes}}
\newcommand{\uLte}{\mathop{\underline{\otimes}}\limits^{\rm L}} 
\newcommand{\Lp}{\mathrm{L}p}
\newcommand{\pder}[3][]{\frac{\partial^{#1}#2}{\partial{#3}}}
\newcommand{\reg}{\mathrm{reg}}
\newcommand{\sing}{\mathrm{sing}}
\newcommand{\fExt}{\mathcal{E}xt}
\newcommand{\fTor}{\mathcal{T}or}
\newcommand{\fEnd}{\mathcal{E}nd}
\newcommand{\dL}{\mathrm{L}}
\newcommand{\fgd}{\mathrm{fgd}}
\newcommand{\Zl}{Z}
\newcommand{\subtageq}[1]{\stepcounter{equation} \tag*{$(\arabic{section}.\arabic{equation})_#1$}}
\newcommand{\wtmu}{\widetilde{\mu}}
\newcommand{\lexp}{\,^{l}\!}

\title{Holomorphic Frobenius actions for DQ-modules}
\author{Fran\c{c}ois Petit\footnote{The author has been fully supported in the frame of the OPEN scheme of the Fonds National de la Recherche (FNR) with the project QUANTMOD O13/570706}}
\date{}

\maketitle

\begin{abstract}
Given a complex manifold endowed with a $\C^\times$-action and a DQ-algebra equipped with a compatible holomorphic Frobenius action (F-action), we prove that if the $\C^\times$-action is free and proper, then the category of F-equivariant DQ-modules is equivalent to the category of modules over the sheaf of invariant sections of the DQ-algebra. As an application, we deduce the codimension three conjecture for formal microdifferential modules from the one for DQ-modules on a symplectic manifold.
\end{abstract}
\section{Introduction}
Relying on the notion of Frobenius action for Deformation Quantization modules (DQ-modules) introduced in \cite{KR}, we establish an equivalence between the category of coherent Frobenius equivariant DQ-modules over a DQ-algebra and the category of modules over the sheaf of invariant sections of this DQ-algebra. This result applied to the special case of the canonical DQ-algebra $\chW$ on the cotangent bundle  provides an equivalence between coherent $\mathrm{F}$-equivariant $\chW$-modules and coherent microdifferential modules on the projective cotangent bundle. This equivalence permits to deduce the codimension three conjecture for formal microdifferential modules \cite{KV} from the one for DQ-modules on a symplectic manifold \cite{codim}.

Deformation Quantization algebras (DQ-algebras) are non-commutative formal deformations of the structure sheaf of a complex variety. They are used to quantize complex Poisson varieties. In the symplectic case, they are often presented as an extension of the ring of microdifferential operators to arbitrary symplectic manifolds. The ring of formal microdifferential operators $\chE$, introduced in \cite{SKK}, is a sheaf on the cotangent bundle of a complex manifold that quantizes it as a \textit{homogeneous symplectic manifold}. DQ-algebras and in particular the canonical deformation quantization of the cotangent bundle $\chW$ ignore the homogeneous structure and quantize this bundle as a symplectic manifold. Using DQ-algebroid stacks, a generalization of the notion of DQ-algebra, this allows one to produce quantizations of arbitrary complex symplectic manifolds using $\chW$ (see \cite{PS}) and in some sense extends formal microdifferential modules to arbitrary symplectic manifolds (Note that it is always possible to quantize complex Poisson varieties as proved in \cite{Cal,Yek} building upon ideas of Kontsevich \cite{Kos}).

The ring $\chW$ is an $\chE$ algebra. Hence, it is natural to ask if it possible to identify those $\chW$-modules which are extension of $\chE$-modules. For that purpose, it is necessary to add an extra structure to encode the compatibility with the $\C^\times$-action on the fibers of the cotangent bundle. This can be achieved by using the notion of \textit{holomorphic Frobenius action}. It was introduced by Masaki Kashiwara and Rapha\"el Rouquier in their seminal work \cite{KR} which introduced an analogue of Beilinson-Bernstein's localization for rational Cherednik algebras. Objects originating from deformation quantization are defined over the ring of formal power series $\C[[\hbar]]$ or its localization with respects to $\hbar$ that is the field of formal Laurent series $\C((\hbar))$. This makes these objects too large for many applications since what is often required is an object satisfying certain finiteness assumptions over $\C$. To overcome this difficulty, they introduced, in \cite{KR}, the notion of $\chW$-algebra with a holomorphic Frobenius action or $\mathrm{F}$-action for short. Given a complex symplectic manifold $X$ endowed with an action of $\C^\times$  and quantized by a DQ-algebra, an $\mathrm{F}$-action is a compatible action of $\C^\times$ on  the DQ-algebra, acting on the deformation parameter $\hbar$ with a weight. This allows one to rescale $\chW$ and the $\chW$-modules with respect to $\hbar$. These actions have been subsequently used by several authors in problems arising from the study of the representation theory of quantized conic symplectic singularities, and in particular rational Cherednik algebras (see for instance \cite{fourcell,Bellhyptor,BLPW,Los1,Los2,mcgerty})

In this paper, we study the notion of DQ-modules endowed with an $\mathrm{F}$-action. The definition of an $\mathrm{F}$-action initially provided by Kashiwara and Rouqier is a punctual definition which makes it difficult to use for problems of global nature as questions of analytic extension (i.e. extending an $\mathrm{F}$-action through an analytic subset). Hence, we provide a reformulation in the spirit of $G$-linearization of coherent sheaves (see \cite[Ch.1 \S 3]{GIT}). Given a DQ-algebra $\cA_X$, on a Poisson manifold $X$, endowed with an $\mathrm{F}$-action, and assuming that this action is free and proper, we establish an equivalence between the category of coherent DQ-modules endowed with an $\mathrm{F}$-action and the category of modules over the sheaf of invariant sections on the quotient space $Y=X / \C^\times$ (Theorem \ref{thm:equimicroDQ}). Here we have to work on the quotient space since $\C^\times$ is not simply connected and $\mathrm{F}$-equivariant DQ-modules are constant along the orbits. Our result generalizes the first example of \cite[\S 2.3.3]{KR} (provided without a proof) which states an equivalence of categories between good $\chW$-modules and good micro-differential modules. We extend this example to DQ-modules over arbitrary Poisson manifolds and relax the finiteness conditions by only requiring the DQ-modules to be coherent. To obtain this equivalence of categories, we first prove that a locally finitely generated $\cA_X$-module endowed with an $\mathrm{F}$-action is locally finitely generated by locally invariant sections (Theorem \ref{thm:invariantgeneration}). This implies that if $\cM$ is coherent, it locally has an equivariant presentation of length one (Corollary \ref{cor:geninvglob}). We prove that the invariant sections functor and the equivariant extension functor form an adjoint pair (Proposition \ref{prop:adjunction}) and establish the coherence of the sheaf of invariant sections (Theorem \ref{thm:cohinv}). Then we can prove the equivalence announced earlier (Theorem \ref{thm:equimicroDQ}). As an example, we construct the weight one $\mathrm{F}$-action on the canonical deformation quantization $\chW$ of the cotangent bundle and obtain as a corollary of Theorem \ref{thm:equimicroDQ} an equivalence between coherent $\chW$-modules and coherent formal microdifferential modules on the projective cotangent bundle (see Proposition \ref{prop:exemple} for a precise statement). Finally, we use this result to deduce the codimension three conjecture for formal microdifferential modules initially proved  by Kashiwara and Vilonen (in the formal as well as in the analytic case) in \cite{KV} from its DQ-module version proved in \cite{codim}. For that purpose, we have to extend the $\mathrm{F}$-action through analytic subsets, which is one of the reason, we defined $\mathrm{F}$-actions in a non-punctual manner.
\vspace{0.3cm}

\noindent \textbf{Acknowledgements.}
I would like to express my gratitude to  Masaki Kashiwara and Pierre Schapira for their scientific insights. It is pleasure to thanks Gwyn Bellamy, Damien Calaque, Vincent Pecastaing, Mauro Porta, Marco Robalo, Yannick Voglaire for useful conversations.
\section{Preliminaries on DQ-modules}

We write $\C^\hbar$ for the ring of formal power series in $\hbar$ with complex coefficients and $\C^{\hbar,loc}$ for the field of formal Laurent series. Let $(X,\cO_X)$ be a complex manifold. We define the sheaf of $\C^\hbar$-algebras
\begin{equation*}
\cO_X^\hbar:=\varprojlim_{n \in \N} \cO_X \te_\C (\C^\hbar / \hbar^n \C^\hbar).
\end{equation*}

\begin{definition}
	A star-product denoted $\star$ on $\cO_X^\hbar$ is a $\C^\hbar$-bilinear associative multiplication law satisfying
	\begin{equation*}
	f \star g = \sum_{i \geq 0} P_i(f,g) \hbar^i \;\; \textnormal{for every} \;f, \; g \in \cO_X,
	\end{equation*}
	where the $P_i$ are holomorphic bi-differential operators such that for every $f, g \in \cO_X, P_0(f,g)=fg$ and 
	$P_i(1,f)=P_i(f,1)=0$ for $i>0$. The pair $(\cO_X^\hbar, \star)$ is called a star-algebra. 
\end{definition}

\begin{definition}
	A DQ-algebra $\cA_X$ on $X$ is a $\C_X^\hbar$-algebra locally isomorphic to a star-algebra as a $\C_X^\hbar$-algebra.
\end{definition}

\begin{notation}
	\begin{enumerate}[(i)]
		\item If $\cA_X$ is a DQ-algebra, we set $\cA_X^{loc}:=\C^{\hbar,loc} \te_{\C^\hbar} \cA_X$,
		
		\item  if $X$ and $Y$ are two complex manifolds endowed with DQ-algebras $\cA_X$ and $\cA_Y$ then $X \times Y$ is canonically equiped with a DQ-algebra $\cA_{X \times Y} := \cA_X \ubtimes \cA_Y$ (see \cite[\textsection 2.3]{KS3}). There is a canonical morphism of $\C^\hbar$-algebras
		\begin{equation*}
		p_2^\sharp \colon p^{-1}_2\cA_X \to \cA_X \boxtimes \cA_Y \to \cA_{X \times Y}
		\end{equation*}
		and this morphism is flat (\cite[lemma 2.3.2]{KS3}).
		\item We denote by $\Mod(\cA_X)$ the Grothendieck category of left $\cA_X$-modules, by $\Mod_\coh(\cA_X)$ its full abelian subcategory whose objects consist of coherent $\cA_X$-modules. We use similar notation for the left $\cA_X^{loc}$-modules.
	\end{enumerate}
\end{notation}

There is a unique isomorphism $\cA_X / \hbar \cA_X \stackrel{\sim}{\longrightarrow} \cO_X$ of $\C_X$-algebras. We denote by $\sigma_0: \cA_X \twoheadrightarrow \cO_X$ the epimorphism of $\C_X$-algebras defined as the following composition
\begin{equation*}
\cA_X \to \cA_X / \hbar \cA_X \stackrel{\sim}{\longrightarrow} \cO_X.
\end{equation*}
These data induce a Poisson bracket $\lbrace \cdot,\cdot \rbrace$ on $\cO_X$ defined by:
\begin{equation*}
\textnormal{for every $a, \; b \in \cA_X$}, \; \lbrace \sigma_0(a),\sigma_0(b) \rbrace=\sigma_0(\hbar^{-1}(ab-ba)).
\end{equation*}

\begin{lemma}
	Let $(\cO_X^\hbar, \star)$ be a star algebra and $v : \cO^\hbar_X \to \cO^\hbar_X$ be a $\C$-linear derivation of $(\cO_X^\hbar, \star)$ such that there exists $v_0 \in \End_{\C_X}(\cO_X)$ such that for every $u \in \cO_X^\hbar$, $\sigma_0 \circ v (u) = v_0 \circ \sigma_0 (u)$ and $v(\hbar)=m\hbar$. Then, $v_0$ is a $\C$-linear derivation of $\cO_X$ and there exists a unique sequence $(v_k)_{k \geq 1}$ of differential operators such that for any $f \in \cO_X$,%
	\begin{equation*}
	v(f)=\sum_{i \geq 0} \hbar^i v_i(f).
	\end{equation*}
	
	In particular, for every $u=\sum_i \hbar^i u_i \in \cO^\hbar_X$,
	\begin{align}\label{eq:basederivation}
	v(u)&=\sum_i \left(\sum_k \hbar^{i+k} v_k(u_i) + m \, i \, \hbar^i u_i \right)\\
	&=\sum_n \hbar^n \left( \sum_{i+k=n} v_k(u_i) + m \,n \, u_n \right).
	\end{align}
\end{lemma}
\begin{proof} This proof is an adaptation of the proof of \cite[Proposition 2.2.3]{KS3}. It is clear that there exists a sequence $(v_k)_{k \geq 0}$  of endomorphism of $\cO_X$ such that, for every $f \in \cO_X$
	\begin{equation*}
	v(f)=\sum_i \hbar^i v_i(f). 
	\end{equation*}
	
	Let $(P_n)_{n \in \N}$ be the sequence of bidifferential operators associated with the star products $\star$. By assumption $v$ is continuous for the $\hbar$-adic topology, thus for every $f$, $g \in \cO_X$,
	\begin{equation*}
	v(f \star g)=\sum_{j \geq 0}  v(\hbar^j P_j(f,g)) =\sum_{n \geq 0} \hbar^n \left( \sum_{i+j=n} v_i(P_j(f,g))+mn \, P_n(f,g) \right)
	\end{equation*}
	and
	\begin{equation*}
	f \star v(g) +v(f) \star g=\sum_{n\geq 0} \hbar^n \sum_{j+k=n} \left( P_k(f,v_j(g)) + P_k(v_j(f),g) \right).
	\end{equation*}
	Since $v(f \star g)=f \star v(g) + v(f) \star g$, we obtain 
	\begin{equation}\label{form:recdiffop}
	\sum_{i+j=n} v_i(P_j(f,g))+mn \, P_n(f,g)=\sum_{j+k=n} \left( P_k(f,v_j(g)) + P_k(v_i(f),g) \right).
	\end{equation}
	For $n=0$, the above formula implies that
	\begin{equation*}
	v_0(fg)=fv_0(g)+v_0(f)g.
	\end{equation*}
	Hence, $v_0$ is a $\C$-linear derivation of $\cO_X$ and in particular is a differential operator.
	
	We will prove by induction that the $v_k$ are differential operators. We just checked that this is true for $k=0$. Assume that this is true for $k<l$ with $l \in \N$. Using the induction hypothesis, we deduce from the expression \eqref{form:recdiffop} that
	\begin{equation*}
	v_l(fg)+Q_l(f,g)= f v_l(g)+v_l(f)g+R_l(f,g)
	\end{equation*}
	where $Q_l$ and $R_l$ are bidifferential operators. This implies that
	\begin{equation*}
	[v_l,g](f)=v_l(fg)-g\,v_l(f)=f \, v_l(g)-Q_l(f,g)+R_l(f,g).
	\end{equation*}
	Since $Q_l(\cdot,g)$ and $R_l(\cdot,g)$ are differential operators, it follows from \cite[Lemma 2.2.4]{KS3} that $v_l$ is a differential operator.
\end{proof}

\subsection{The canonical deformation quantization of the cotangent bundle}

Let $M$ be a complex manifold. The cotangent bundle of $M$, $X:=T^\ast M$ is equipped with the sheaf  $\chE_X$ of formal microdifferential operators. This is a filtered, conic sheaf of $\C$-algebras. We denote by  $\chE_X(0)$ the subsheaf of $\chE_X$ formed by the operators of order $m \leq 0$. These sheaves were introduced in \cite{SKK}. The reader can consult \cite{Sch} for an introduction to the theory of microdifferential modules.

On $X$, there is a DQ-algebra $\chW_X(0)$ which was constructed in \cite{PS}. Here, we review their construction. 

Let $\C$ be the complex line endowed with the coordinate $t$ and denote by $(t;\tau)$ the associated symplectic coordinate on $T^\ast \C$. We set
\begin{equation*}
\chE_{T^\ast ( M \times \C), \hat{t}}(0)=\lbrace P \in \chE_{T^\ast  M}; [P, \partial_t]=0 \rbrace.
\end{equation*}
We consider the following open subset of $T^\ast (M \times \C)$
\begin{equation*}
T_{\tau \neq 0}^\ast(M\times \C)=\lbrace (x,t;\xi,\tau) \in T^\ast (M \times \C) | \tau \neq 0 \rbrace
\end{equation*}
and the morphism
\begin{equation*}
\rho \colon T_{\tau \neq 0}^\ast(M\times \C) \to T^\ast M, \; (x,t;\xi,\tau) \mapsto (x;\xi/ \tau). 
\end{equation*}
We obtain the $\C_X^\hbar$-algebra
\begin{equation}\label{def:WO}
\chW_X(0) \colon=\rho_\ast (\chE_{T^\ast ( M \times \C), \hat{t}}(0)|_{ T_{\tau \neq 0}^\ast(M\times \C)})
\end{equation}
where $\hbar$ acts as $\tau^{-1}$.
A section $P$ of $\chW_X(0)$ can be written in a local symplectic coordinate system $(x_1,\ldots,x_n,u_1,\ldots,u_n)$ as
\begin{equation*}
P=\sum_{j \leq 0} f_{-j}(x,u_i) \tau^j, \, f_{-j} \in \cO_X,\; j \in \Z.
\end{equation*}
Setting $\hbar=\tau^{-1}$, we obtain
\begin{equation*}
P=\sum_{k \geq 0} f_k(x,u_i) \hbar^k, \, f_k \in \cO_X,\; k \in \N.
\end{equation*}

We write $\chW_X$ for the localization of $\chW_X(0)$ with respect to the parameter $\hbar$. There is the following commutative diagram of morphisms of algebras.
\begin{equation*}
\xymatrix{
	\chE_X \ar@{^{(}->}[r]^-{\iota} & \chW_X\\
	\chE_X(0) \ar@{^{(}->}[u] \ar@{^{(}->}[r] & \chW_X(0) \ar@{^{(}->}[u]
}
\end{equation*}
where the algebra map $\iota:\chE_X \to  \chW_X$ is given in a local symplectic coordinate system $(x_1,\ldots,x_n,u_1,\ldots,u_n)$ by $x_i \mapsto x_i$, $\partial_{x_i} \mapsto  \hbar^{-1} u_i$.

\section{Sections depending on a complex parameter}
Let $X$ be a complex manifold endowed with a DQ-algebra $\cA_X$.
We consider the DQ-algebra $\cA_{\C \times X}= \cO_\C^\hbar \underline{\boxtimes} \cA_X $ on $\C \times X$. We denote by $t$ the coordinate on $\C$, by $p_2 \colon \C \times X \to X$ the projection on $X$ and by $p_1$ the projection on $\C$. Note that $\cA_{\C \times X}$ is a left $\cD_\C$-modules and in particular a left $\cO_\C$-module.
Let $t \in \C$, denote by $\mathfrak{m}_t$ the maximal ideal of $\cO_{\C,t}$ and consider the morphism
\begin{equation*}
i_t \colon X \to \C \times X, \; x \mapsto (t,x).
\end{equation*} 
Then, we have an evaluation morphism
\begin{align*}
ev_t\colon i_t^{-1}\cA_{\C \times X} &\to i_t^{-1}\cA_{\C \times X}/ \mathfrak{m}_t(i_t^{-1}\cA_{\C \times X}) \simeq \cA_X\\
u &\mapsto u(t).
\end{align*}
and
\begin{align*}
ev_t\colon i_t^{-1}\cA^{loc}_{\C \times X} & \to i_t^{-1}\cA^{loc}_{\C \times X}/ \mathfrak{m}_t (i_t^{-1}\cA^{loc}_{\C \times X})\simeq \cA^{loc}_X\\
u &\mapsto u(t).
\end{align*}

\begin{notation}
	\begin{enumerate}[(i)]
		\item Let $(f,f^\sharp)\colon (X,\cR_X) \to (Y,\cR_Y)$ be a morphism of ringed spaces. As usual, we denote by $f^\ast$ the functor
		\begin{equation*}
		f^\ast \colon\Mod(\cR_Y) \to \Mod(\cR_X), \; \cM \mapsto f^\ast \cM \colon \!\!\!= \cR_X \te_{f^{-1}\cR_Y} f^{-1}\cM. 
		\end{equation*}
		
		\item In order to keep the number of notations to a bearable level, we will write indistinctly $p_2^\ast \cM$ for $\cA_{\C \times X} \te_{p_2^{-1}\cA_X} p_2^{-1} \cM$ and for $\cA^{loc}_{\C \times X} \te_{p_2^{-1}\cA^{loc}_X} p_2^{-1} \cM$ depending of whether $\cM$ is considered as an $\cA_X$-module or an $\cA_X^{loc}$-module. 
	\end{enumerate}
\end{notation}

\begin{definition}
	Let $\cM$ be an $\cA_X$-module (resp. an $\cA_X^{loc}$-module) and set $\cN=p_2^\ast \cM$ and consider $s\in \cN$. The module $\cN$ is a $\cD_\C$-module. The derivative with respect to $t$ of a section $s$ in $\cN$ is the section $\partial_t s$. It is denoted $s^\prime$ and called the derivative of $s$.
\end{definition}

\begin{definition}
	Let $U$ be an open subset of $\C$ and let $\cM$ be a coherent $\cA_X$-module (resp. $\cA_X^{loc}$-module). Let $(s(t))_{t \in U}$ be a family of section of $\cM$. We say that  $(s(t))_{t \in U}$ depends holomorphically on $t$, if locally there exists a section $s \in p_2^\ast \cM$ such that $ev_t(s)=s(t)$.
\end{definition}

\begin{proposition}\label{prop:annulationcom}
	Let $X$ be a complex manifold and $\cF$ be a coherent $\cO_X$-module on $X$, $U$ an open subset of $\C \times X$ and $u \in p_2^\ast \cF(U)$ such that for every $t \in p_1(U), \; u(t)=0$. Then $u=0$.
\end{proposition}

\begin{proof}
	This question is local. So, we can assume that we are working in the vicinity of a point $(t_0,x) \in \C \times X$. We identify the local ring $(\cO_{X,x},\mathfrak{m}_x)$ with a subring of the local ring $(\cO_{\C \times X, (t_0,x)}, \mathfrak{m}_{(t_0,x)})$ via the morphism of locally ringed spaces induced by the projection $p_2 \colon \C \times X \to X$ . We denote by $\mathfrak{r}_{(t_0,x)}$ the ideal of $\cO_{\C \times X, (t_0,x)}$ generated by $\mathfrak{m}_x$. For every $q \in \N$, we have
	\begin{align*}
	(p_2^\ast \cF)_{(t_0,x)} / \mathfrak{r}_{(t_0,x)}^q (p_2^\ast \cF)_{(t_0,x)} \simeq \cO_{\C,t_0} \otimes \cF_x / \mathfrak{m}^q_x \cF_x.
	\end{align*}
	Writing $u_{t_0}(x)$ for the image of $u$ in $(p_2^\ast \cF)_{(t_0,x)} / \mathfrak{r}_{(t_0,x)}^q (p_2^\ast \cF)_{(t_0,x)}$ and choosing an isomorphism $\cF_x / \mathfrak{m}^q_x \cF_x \simeq \C^r$, we can identify $u_{t_0}(x)$ with a vector $(f_1,\ldots,f_r)$ where the $f_i \in \cO_{\C,t_0}$. It follows from the assumption that there exists a neighborhood $V$ of $t_0$ such that for every $t \in V, \; f_i(t)=0$. This implies that $u_{t_0}(x)=0$ that is $u_{(t_0,x)} \in \mathfrak{r}_{(x,t_0)}^q (p_2^\ast \cF)_{(x,t_0)}$. As $\mathfrak{r}_{(x,t_0)} \subset \mathfrak{m}_{(x,t_0)}$ and $(p_2^\ast \cF)_{(t_0,x)}$ is a finitely generated $\cO_{\C \times X, (t_0,x)}$-module, it follows from the Krull intersection lemma that $u_{(t_0,x)}=0$.  
\end{proof}

\begin{proposition}\label{prop:annulation} 
	Let $\cM$ be a coherent $\cA_X$-module, set $\cN=p_2^\ast \cM$ , $U$ an open subset of $\C \times X$ and let $ u \in \cN(U)$ such that for every $t \in p_1(U)$, $u(t)=0$. Then $u=0$.
\end{proposition}

\begin{proof}
	The $\cO_{\C \times X}$-modules $\hbar^n \cN/ \hbar^{n+1} \cN$ are coherent and
	\begin{equation}\label{eq:reducmodulon}
	\hbar^n\cN / \hbar^{n+1}\cN \simeq p_2^\ast (\hbar^n\cM/ \hbar^{n+1}\cM).
	\end{equation}
	Let $u_0$ be the image of $u$ via the map $\cN \to \cN/ \hbar \cN$. It follows from the assumptions that for every $t \in p_1(U)$, $u_0(t)=0$ and from the isomorphism \eqref{eq:reducmodulon} that $u_0 \in p_2^\ast (\cM/ \hbar\cM)$. Then by  Proposition \ref{prop:annulationcom}, $u_0=0$. That is $u \in \hbar \cN$. 
	
	Let us show by recursion that  $u \in \bigcap_{n \geq 0} \hbar^n \cN$. We just proved that $u \in \hbar \cN$. Assume that $u \in \hbar^n\cN$ and denote by $u_n$ the image of $u$ via the map $\cN \to \hbar^n \cN/ \hbar^{n+1} \cN$. By the isomorphism \eqref{eq:reducmodulon} we identify $u_n$ with a section of the coherent $\cO_{\C \times X}$-module $p_2^\ast (\hbar^n\cM/ \hbar^{n+1}\cM)$ such that for every $t$, $u_n(t)=0$. Thus by Proposition \ref{prop:annulationcom}, $u_n=0$ that is $u \in \hbar^{n+1} \cN$. It follows that $u \in \bigcap_{n \geq 0} \hbar^n \cN$ and $\bigcap_{n \geq 0} \hbar^n \cN=(0)$ by \cite[Corollary 1.2.8]{KS3} which proves the claim.
\end{proof}

\begin{corollary}\label{cor:annulation}
	Let $\cM$ be a coherent $\cA^{loc}_X$-module and let $\cN=p_2^\ast \cM$. Let $U$ be an open subset of $\C \times X$ and $ u \in \cN(U)$ such that for every $t \in p_1(U)$, $u(t)=0$. Then $u=0$.
\end{corollary}

\begin{proof}
	Let $(t,x) \in U$. There exist an open neighborhood $V \times W \subset U$ of $(t,x)$ and finitely many $u_i \in \cM|_W$ such that $\cM|_W=\sum_i \cA^{loc}_W u_i$. We consider the $\cA_W$-module $\cM^\prime=\sum_i \cA_W u_i$. It is a finitely generated $\cA_W$-submodule of the coherent $\cA_W^{loc}$-module $\cM$. Thus, $\cM^\prime$ is coherent. Shrinking $V \times W$ if necessary and multiplying $u$ by $\hbar^n$ with $n \in \N$ sufficiently big, we can assume that $\hbar^n u \in \cA_{V \times W} \otimes_{\cA_W} \cM^\prime$. The section $\hbar^n u$ satisfies the hypothesis of the Proposition \ref{prop:annulation}. It follows that $\hbar^n u=0$. But, the action of $\hbar$ on $\cN$ is invertible. It follows that $u=0$.   
\end{proof}

\begin{corollary}\label{cor:loccons}
	Let $\cM$ be a coherent $\cA_X$-module (resp. $\cA^{loc}_X$-module) and set $\cN= p_2^{\ast}\cM$. Let  $U$ an open subset of $\C \times X$ and $ u \in \cN(U)$ such that for every $t \in p_1(U)$, $u^\prime(t)=0$. Then $u \in p_2^{-1}\cM(U)$.
\end{corollary}
\begin{proof}
	Since $\cM$ is coherent, locally it has a presentation
	\begin{equation*}
	0 \to \cI \to \cA_X^m \to \cM \to 0.
	\end{equation*}
	Since $\cA_{\C \times X}$ is flat over $p_2^{-1}\cA_X$, the module $\cN$ has the following presentation
	\begin{equation}
	0 \to \cA_{\C \times X}\cI \to \cA_{\C \times X}^m \to \cN \to 0.
	\end{equation}
	Let $(t_0,x_0) \in \C \times X$. There exist an open neighborhood $V$ of $(t_0,x_0)$ and a section $s=\sum_{i=1}^n a_i e_i \in \cA_{\C \times X}^m|_V$ such that its image in $\cN$ is $u$. 
	
	By hypothesis $u^\prime(t)=0$, it follows from Proposition \ref{prop:annulation} (resp. Corollary \ref{cor:annulation}) that $u^\prime=0$ which implies that we can write
	\begin{equation*}
	s^\prime=\sum_j b_jv_j
	\end{equation*}
	with $b_j \in \cA_{\C \times X}$ and $v_j \in \cI$. Let $c_j$ be a primitive of $b_j$ in a neighborhood of $t_0$ and set $w=\sum_j c_j \otimes v_j$. Thus $(s-w)^\prime=0$ in $\cA^m_{\C \times X}$ which implies that $s-w \in p_2^{-1}\cA_{X}^m$. Finally since $s-w$ and $s$ have  the same image in $\cN$, it follows that $u$ does not depend on $t$ i.e $u \in p_2^{-1}\cM(U)$.
\end{proof}

\section{Holomorphic Frobenius actions}

In this section, we precise certain aspects of the definition of an $\mathrm{F}$-action on a DQ-algebra or on a DQ-module. This notion was introduced in \cite[Definition 2.2 and Definition 2.4]{KR}.

Let $(X,\{ \cdot\, , \, \cdot \})$ be a complex Poisson manifold. We assume that it comes equipped with a torus action, $\C^\times\to \Aut(X), \; t \mapsto \mu_t$ such that $\mu^\ast_t\{f,\,g\}= t^{-m}\{\mu^\ast_t f, \, \mu^\ast_t g \} $ with $m \in \Z^\ast$.

\begin{notation}
	\begin{itemize}
		\item We denote by $\sigma \colon \C^\times \times \C^\times \to \C^\times$ the group law of $\C^\times$,
		
		\item $\mu  \colon \C^\times \times X   \to X$ the action of $\C^\times$ on $X$.
		
		\item $\wtmu \colon \C^\times \times X \to \C^\times \times X, \; (t,x) \mapsto (t,\mu(t,x))$
		
		\item for $t \in \C^\times$, the morphism
		\begin{equation*}
		i_t \colon X \to \C^\times \times X, \; x \mapsto (t,x).
		\end{equation*}
		
		\item We write $\wtmu_t$ for the composition $\wtmu \circ i_t$ (Note that $\mu_t=\mu \circ i_t$).
		
		\item Consider the product of manifolds $\C^\times \times X$. We denote by
		$p_i$ the $i$-th projection. 
		
		\item Consider the product of manifolds $\C^\times \times \C^\times \times X$. We denote by
		$q_i$ the $i$-th projection, and by $q_{ij}$ the $(i,j)$-th projection
		({\em e.g.,} $q_{13}$ is the projection from 
		$ \C^\times \times \C^\times \times X$ to $\C^\times \times X$, $(t_1,t_2,x_3) \mapsto (t_1,x_3)$). 
		\item We write $a_X \colon X \to \pt$ for the unique map from $X$ to the point.
		\item Recall that, in all this paper, $\cA_X$ is DQ-algebra and we write $\cA_{\C^\times \times X}$ for the DQ-algebra $ \cO_{\C^\times}^\hbar \ubtimes \cA_X $.
	\end{itemize}
\end{notation}

\begin{lemma}\label{lem:extensioncompletion}
	Let $\widetilde{\theta}:\wtmu^{-1} \cA_{\C^\times \times X} \to \cA_{\C^\times \times X}$ be a morphism of sheaves of $p_1^{-1}\cO_{\C^\times}$-algebras such that the adjoint morphism $\psi: \cA_{ \C^\times  \times X} \to \wtmu_\ast\cA_{\C^\times \times X}$ is a continuous morphism of Fr\'echet $\C$-algebras. Then the dashed arrow in the below diagram is filled by a unique morphism $\widetilde{\lambda}$ of  $q_{12}^{-1}\cO_{\C^\times \times \C^\times}$-algebras. If $\widetilde{\theta}$ is an isomorphism then $\widetilde{\lambda}$ also.
	
	\begin{equation*}
	\xymatrix{
		( \id_{\C^\times} \times \wtmu )^{-1}(\cO_{\C^\times} \boxtimes \cA_{\C^\times  \times X}) \ar[r]^-{ \id \times \widetilde{\theta}} \ar@{^{(}->}[d] & \cO_{\C^\times} \boxtimes \cA_{\C^\times \times X} \ar@{^{(}->}[d]\\
		( \id_{\C^\times} \times \wtmu )^{-1}\cA_{ \C^\times \times \C^\times \times X } \ar@{-->}[r]_-{\widetilde{\lambda}}& \cA_{\C^\times \times \C^\times \times X}
	}
	\end{equation*}
\end{lemma}

\begin{proof}
	By adjunction, it is equivalent to show that the dashed arrow in the below diagram is filled by a unique map of $q_{12}^{-1}\cO_{\C^\times \times \C^\times}$-algebras.

	\begin{equation*}
	\xymatrix{
		(\cO_{\C^\times} \boxtimes \cA_{\C^\times  \times X}) \ar[r]^-{ \id \times \psi} \ar@{^{(}->}[d] & ( \id_{\C^\times} \! \!\times \wtmu )_{\ast}(\cO_{\C^\times} \boxtimes \cA_{\C^\times \times X})   \ar@{^{(}->}[d]\\
		\cA_{ \C^\times \times \C^\times \times X } \ar@{-->}[r]&( \id_{\C^\times} \!\! \times \wtmu )_{\ast}\cA_{\C^\times \times \C^\times \times X}
	}
	\end{equation*}
	
	Denoting by $\overset{\mathrm{p}}{\boxtimes}$ the external product of presheaves, there is a morphism of presheaves
	\begin{equation}\label{mor:presheafprecomp}
	\id \overset{\mathrm{p}}{\boxtimes} \psi \colon \cO_{\C^\times} \overset{\mathrm{p}}{\boxtimes} \cA_{\C^\times \times X} \to \cO_{\C^\times} \overset{\mathrm{p}}{\boxtimes} \wtmu_\ast \cA_{\C^\times \times X}.
	\end{equation}

	Recall that DQ-algebras, the sheaf $\cO_{\C^\times}$ as well as $\wtmu_{\ast}\cA_{\C^\times \times X}$ are sheaves of nuclear Fr\'echet $\C$-algebras. Moreover, there exists a countable basis $\mathfrak{B}$ of open set of $\C^\times \times \C^\times \times X$ of the form $U_i \times V_j$ such that $\cA_{\C^\times \times X} |_{V_j}$ is isomorphic to a star-algebra. Hence, evaluating the morphism \eqref{mor:presheafprecomp}, on the $U_i \times V_j \in \mathfrak{B}$, we get the continuous morphism of topological $\C$-algebras \eqref{mor:contfre} (As the spaces, we consider are nuclear, the choice of a topology on the tensor products does not matter. For instance, we endow all the tensor product of nuclear spaces with the projective topology.) 
	
	\begin{equation}\label{mor:contfre}
	(\id \otimes \psi)_{U_i \times V_j} \colon \cO_{\C^\times}(U_i) \otimes_\pi \cA_{\C^\times \times X} (V_j) \to \cO_{\C^\times}(U_i) \otimes_\pi \wtmu_\ast\cA_{\C^\times \times X}(V_j).
	\end{equation}
	
	By definition the morphisms $\id_{U_i} \! \otimes  \psi_{V_j}$ are compatible with restrictions and applying the completion functor to the above morphisms, we get the following diagram
	
	\begin{equation*}
	\xymatrix{
		\cO_{\C^\times}(U_i) \otimes_\pi \cA_{\C^\times \times X} (V_j) \ar[r]^-{\id \otimes \psi} \ar[d] & \cO_{\C^\times}(U_i) \otimes_\pi \wtmu_\ast\cA_{\C^\times \times X}(V_j) \ar[d]\\
		\cO_{\C^\times}(U_i) \widehat{\otimes}_\pi \cA_{\C^\times  \times X} (V_j)\ar[r]^-{\id \widehat{\otimes} \psi}& \cO_{\C^\times}(U_i) \widehat{\otimes}_\pi \wtmu_\ast\cA_{\C^\times \times X}(V_j).
	}
	\end{equation*}
	We have obtained a family of morphisms of Fr\'echet algebras $\lbrace \id_{U_i} \widehat{\otimes} \psi_{V_j} \rbrace_{U_i \times V_j \in \mathfrak{B}}$. 
	
	We describe the completion of the topological vector spaces 
	\begin{align*}
	\cO_{\C^\times}(U_i) \otimes_\pi \cA_{\C^\times \times X} (V_j) && \cO_{\C^\times}(U_i) \otimes_\pi \wtmu_\ast\cA_{\C^\times \times X}(V_j).
	\end{align*}
	
	Observe that, on $V_j$, there is an isomorphism of Fr\'echet algebra
	\begin{equation*}
	\cA_{\C^\times \times X}(V_j) \simeq \cO^\hbar_{\C^\times \times X}(V_j).
	\end{equation*}
	Hence, we obtain a continuous inclusion with dense image
	\begin{equation*}
	\cO_{\C^\times}(U_i) \otimes_\pi \cO^\hbar_{\C^\times \times X} (V_j) \hookrightarrow \prod \cO_{\C^\times}(U_i) \otimes_\pi \cO_{\C^\times \times X}(V_j).
	\end{equation*}
	Applying the completion functor and using the fact that it commutes with products, we obtain the following isomorphisms algebras
	\begin{equation*}
	\cO_{\C^\times}(U_i) \widehat{\otimes}_\pi \cO^\hbar_{\C^\times \times X} (V_j) \stackrel{\sim}{\to} \prod \cO_{\C^\times}(U_i) \widehat{\otimes}_\pi \cO_{\C^\times \times X}(V_j)\simeq \cA_{\C^\times \times \C^\times \times X}(U_i \times V_j).
	\end{equation*}
	Similarly, we have that
	\begin{equation*}
	\cO_{\C^\times}(U_i) \widehat{\otimes}_\pi \wtmu_\ast\cA_{\C^\times \times X}(V_j) \simeq (\id_{\C^\times} \times \wtmu)_\ast \cA_{\C^\times \times \C^\times \times X}(U_i \times V_j).
	\end{equation*}
	Hence, we have obtained a family of morphism of $\C$-algebras $\lbrace \cA_{\C^\times \times \C^\times \times X}(U_i \times V_j) \to (\id \times \wtmu)_\ast \cA_{\C^\times \times \C^\times \times X} (U_i \times V_j) \rbrace_{U_i \times V_j \in \mathfrak{B}}$ which extends to a unique morphism of sheaves on $\C^\times \times \C^\times \times X$, $\cA_{\C^\times \times \C^\times \times X} \to (\id \times \wtmu)_\ast \cA_{\C^\times \times \C^\times \times X}$.
\end{proof}
By \cite[Lemma 2.2.9]{KS3}, there is a canonical morphism
\begin{equation*}
q_{23}^\sharp \colon q_{23}^{-1}\cA_{\C^\times \times X} \to \cA_{\C^\times \times \C^\times \times X}.
\end{equation*}
With notations as in Lemma \ref{lem:extensioncompletion}, we obtain the morphism $\lambda$ as the composition
\begin{equation}\label{mor:lambda}
\lambda \colon (\id \times \mu)^{-1} \cA_{\C^\times \times X} \stackrel{(\id \times \wtmu)^{-1}q_{23}^\sharp}{\longrightarrow} (\id \times \wtmu)^{-1}\cA_{\C^\times \times \C^\times \times X} \stackrel{\widetilde{\lambda}}{\to}\cA_{\C^\times \times \C^\times \times X}.
\end{equation}

We introduce the functor
\begin{align*}
\mathrm{Ev_t} &\colon \Mod(p_1^{-1} \cO_{\C^\times}) \to \Mod(\C_X)\\ 
\hspace{3cm}       & \cM \mapsto a_X^{-1}(\cO_{\C^\times,t} / \mathfrak{m}_t) \!\! \te_{a_X^{-1} \cO_{\C^\times,t}} \! \! i_t^{-1}\cM \simeq i_t^{-1}\cM / a_X^{-1}\mathfrak{m}_t\, i_t^{-1}\cM. 
\end{align*}
In particular, $\mathrm{Ev}_t(\cA_{\C^\times \times X}) \simeq \cA_X$ and $\mathrm{Ev}_t(\wtmu^{-1}\cA_{\C^\times \times X}) \simeq \mu^{-1}_t \cA_X$.\\

The following definition should be compared with \cite[Definition 2.2]{KR} and with \cite[p.15]{BLPW}.
\begin{definition}
	An $\mathrm{F}$-action with exponent $m$ on $\cA_X$ is the data of an isomorphism of $p_1^{-1} \cO_{\C^\times}$-algebras $\widetilde{\theta} \colon \wtmu^{-1}\cA_{\C^\times \times X} \to \cA_{\C^\times \times X}$ such that 
	
	\begin{enumerate}[(a)]
		
		\item the morphism $\theta_t :=\mathrm{Ev_t}(\widetilde{\theta})$ satisifies $\theta_1=\id$,
		
		\item for every $t \in \C^\times$, $\theta_t(\hbar^n)=t^{mn} \hbar^n$,
		
		\item the adjoint morphism of $\widetilde{\theta}$, $\widetilde{\psi}: \cA_{\C^\times \times X} \to \wtmu_\ast\cA_{\C^\times \times X}$ is a continuous morphism of Fr\'echet $\C$-algebras,
		
		\item setting
		\begin{align*}
		\theta \colon \mu^{-1} \cA_X \stackrel{\wtmu^{-1} p_2^{\sharp}}{\longrightarrow} \wtmu^{-1} \cA_{ \C^\times \times X} \stackrel{\widetilde{\theta}}{\to} \cA_{ \C^\times \times X}
		\end{align*}
		the below diagram commutes,
		\begin{equation*}
		\xymatrix{(\id_{\C^\times} \times \mu)^{-1} \mu^{-1} \cA_X \ar[rr]^-{(\id_{\C^\times} \times \mu)^{-1}\theta} \ar@{=}[dd] && (\id_{\C^\times} \times \mu)^{-1} \cA_{\C^\times \times X}  \ar[d]^-{\lambda}\\ 
			&&  \cA_{\C^\times \times \C^\times \times X}\\
			( \sigma \times \id_X)^{-1} \mu^{-1} \cA_X \ar[rr]^-{( \sigma \times \id_X)^{-1} \theta} && ( \sigma \times \id_X)^{-1} \cA_{\C^\times \times X}\ar[u]
		}
		\end{equation*}
		where $\lambda$ is provided by Lemma \eqref{lem:extensioncompletion}. 
		
	\end{enumerate}
\end{definition}
\begin{definition}
	An $\mathrm{F}$-action on $\cA_X^{loc}$ is the localization with respect to $\hbar$ of an $\mathrm{F}$-action on $\cA_X$.
\end{definition}

\begin{remark}
	It would be possible to define directly the notion of $\mathrm{F}$-action on $\cA_X^{loc}$ but the definition would be slightly more involved. Moreover, any such action would be induced by an $\mathrm{F}$-action on $\cA_X$. This justify the choice of our previous definition.
\end{remark}
The pair 
\begin{equation}\label{mor:point}
(i_t,ev_t)\colon  (X, \cA_X^{loc}) \to (\C^\times \times X,\cA^{loc}_{\C^\times \times X}) 
\end{equation}
is a morphism of ringed spaces. The $\mathrm{F}$-action on $\cA_X$ induces another morphism of ringed spaces 
\begin{equation}\label{mor:muextension}
(\mu,\theta) \colon (\C^\times \times X, \cA^{loc}_{\C^\times \times X}) \to (X, \cA^{loc}_X). 
\end{equation}
\begin{remark}
	A word of caution about morphism \eqref{mor:muextension}. This morphism is a morphism of $\C$-ringed spaces but not of $\C^\hbar$-ringed spaces.
\end{remark}
The morphism 
\begin{equation}\label{mor:muid}
\lambda \colon(\id \times \mu)^{-1} \cA^{loc}_{\C^\times \times X} \to \cA^{loc}_{\C^\times \times \C^\times \times X } 
\end{equation}
provided by Lemma \ref{lem:extensioncompletion} and the data of the $\mathrm{F}$-action $\theta$ on $\cA_X^{loc}$ allows to define a morphism of ringed space
\begin{equation*}
(\id \times \mu, \lambda) \colon (\C^\times \times \C^\times \times X,\cA^{loc}_{ \C^\times \times \C^\times \times X}) \to (\C^\times \times X,\cA^{loc}_{\C^\times \times X}).
\end{equation*}
The morphism of sheaves
\begin{equation*}
\sigma^\sharp \colon \sigma^{-1}\cO_{\C^\times} \to \cO_{\C^\times \times \C^\times}
\end{equation*}
induces a map
\begin{equation*}
\alpha \colon (\sigma \times \id_X)^{-1}\cA_{\C^\times \times X} \stackrel{\sigma^\sharp\widehat{\otimes}\id}{\longrightarrow} \cA_{\C^\times \times \C^\times \times X}
\end{equation*}
which provides a morphism of ringed spaces
\begin{equation*}
(\sigma \times \id_X, \alpha) \colon (\C^\times \times \C^\times \times X, \cA_{\C^\times \times \C^\times \times X}) \to (\C^\times \times X,\cA_{\C^\times \times X}).
\end{equation*}
\begin{lemma}
	The morphisms of sheaves of rings $\theta$, $\lambda$ and $\alpha$ are flat.
\end{lemma}

\begin{proof}
	The proof for $\theta$ and $\lambda$ are similar. Hence, we only provide the proof for $\theta$. Since $\widetilde{\theta}$ is an isomorphism it is flat and $\mu^{-1}p_2^\sharp$ is flat by \cite[Lemma 2.3.2]{KS3}. Thus, $\theta= \widetilde{\theta} \circ \mu^{-1}p_2^\sharp$ is flat.
	
	We now prove the flatness of $\alpha$. Since $\sigma$ is a submersion, for every $(t_1,t_2) \in \C^\times \times \C^\times$, there exist an open neighborhood $W$ of $(t_1,t_2)$ and a biholomorphism $g:U\times V \to W$ such that $p_1(z_1,z_2)=\sigma \circ g (z_1,z_2)=z_1$. Since flatness is a local question, we can restrict $\alpha$ to an open neighborhood of the form $W \times W^\prime$ with $W^\prime$ an open subset of $X$. Hence, we obtain the following commutative diagram
	\begin{equation*}
	\xymatrix{
		(p_1 \times \id_{W^\prime})^{-1}\cA_{U \times X} \ar[rr]^-{(g \times \id_{W^\prime})^{-1} \alpha\vert_{W \times W^\prime}} \ar[d]_-{\wr} && \cA_{U \times V \times W^\prime} \ar[d]^-{\wr}\\
		(\sigma|_W \times \id_{W^\prime})^{-1}\cA_{\C^\times \times X} \ar[rr]^-{\alpha\vert_{W \times W^\prime}} && \cA_{\C^\times \times \C^\times \times X}\vert_{W \times W^\prime}
	}
	\end{equation*}
	where the top morphism $(g \times \id_{W^\prime})^{-1} \alpha\vert_{W \times W^\prime}=q_{13}^\sharp$ is flat by \cite[Lemma 2.3.2]{KS3}. This implies that $\alpha$ is flat.
\end{proof}

The following definition is an adaptation of \cite[Definition 2.4]{KR} along the line of \cite[ch.1 \S3 Definition 1.6]{GIT}.

\begin{definition}
	An $\mathrm{F}$-action on a $\cA_X^{loc}$-module $\cM$ is the data of an isomorphism of $\cA^{loc}_{\C^\times \times X}$-modules
	\begin{equation}
	\phi \colon \mu^{\ast} \cM \stackrel{\sim}{\to} p_2^\ast \cM
	\end{equation}
	such that the diagram
	\begin{equation}\label{diag:asso}
	\scalebox{0.85}{\xymatrix{
			(\id_{\C^\times} \times \mu)^{\ast} \mu^{\ast} \cM \ar[rr]^-{(\id_{\C^\times} \times \mu)^{\ast}\phi} \ar@{=}[dd] 
			&& (\id_{\C^\times} \times \mu)^{\ast} p_2^\ast \cM \ar[r]^-{\sim} 
			& q^\ast_{23} \mu^\ast \cM \ar[r]^{q_{23}^\ast \phi} 
			& q_{23}^\ast p_2^\ast \cM \ar[d]^-{\wr}\\ 
			&&&&  q_3^\ast \cM \\
			(\sigma \times \id_X)^{\ast} \mu^{\ast} \cM \ar[rrrr]^-{(\sigma \times \id_X)^{\ast} \phi} 
			&&&& (\sigma \times \id_X)^{\ast} p_2^\ast \cM\ar[u]_-{\wr}
	}}
	\end{equation}
	commutes. 
\end{definition}
Following \cite{KR}, we denote by $\Mod_{\mathrm{F}}(\cA_X^{loc})$ the category of $(\cA_X^{loc},\theta)$-modules whose morphisms  are the morphisms of $\cA_X^{loc}$-modules compatible with the action of $\C^\times$. This category is a $\C$-linear abelian category. We write $\Mod_{\mathrm{F},\coh}(\cA_X^{loc})$ for the full subcategory of $\Mod_{\mathrm{F}}(\cA_X^{loc})$ the objects of which are coherent modules in $\Mod(\cA_X^{loc})$.

Let $\cM$ be an $\cA_X^{loc}$-module endowed with an $\mathrm{F}$-action $\phi \colon \mu^{\ast}\cM \to p_2^\ast \cM$ and $t \in \C^\times$. There is the following commutative diagram defining the morphism $\phi_t$

\begin{equation*}
\xymatrix{
	i_t^\ast \mu^\ast \cM \ar[d]_-{\wr} \ar[r]^-{i_t^\ast \phi} & i_t^\ast p^\ast_2 \cM \ar[d]^-{\wr}\\
	\mu_t^{-1} \cM \ar[r]^-{\phi_t} & \cM
}
\end{equation*}
where the vertical map are isomorphism of $\C_X$-modules. Hence, we have obtained a map of $\C_X$-module
\begin{equation*}
\phi_t \colon \mu_t^{-1} \cM \to \cM
\end{equation*}

such that

\begin{enumerate}[(a)]
	\item $\phi_t$ depends holomorphically of $t$,
	
	\item $\phi_{t t^\prime}=\phi_{t^\prime} \circ \mu^{-1}_{t^\prime}\phi_{t}$ for $t, \, t^\prime \in \C^{\times}$,
	
	\item $\phi_t(a m)= \theta_t(a) \phi_t(m)$ for $a \in \cA_{X}^{loc}$ and $m \in \cM$.
\end{enumerate}

\begin{remark}
	\begin{enumerate}[(a)]
		\item We will usually write $\phi_{t t^\prime}=\phi_{t^\prime} \circ \phi_{t}$ instead of $\phi_{t t^\prime}=\phi_{t^\prime} \circ \mu^{-1}_{t^\prime}\phi_{t}$.
		
		\item This implies that an $\mathrm{F}$-action in our sense gives rise to an $\mathrm{F}$-action in the sense of \cite{KR}. In practice, the examples of $\mathrm{F}$-action in the sense of \cite{KR} are also $\mathrm{F}$-action in our sense. 
	\end{enumerate}
\end{remark}

Let $\cM$ be an $\cA_X^{loc}$-module endowed with an $\mathrm{F}$-action $\phi \colon \mu^{\ast}\cM \to p_2^\ast \cM$. The $\mathrm{F}$-action provides a derivation of $\cM$. Indeed, notice that $p_2^\ast \cM$ has a structure of left $p_2^{-1}\cD_{\C^\times}$-module. Hence, we have

\begin{equation}\label{mor:prederiv}
\xymatrix{
	\mu^{\ast}\cM \ar[r]^-{\phi} & p_2^{\ast} \cM \ar[r]^-{\partial_t} &  p_2^{\ast} \cM
}.
\end{equation}

Let $t_0 \in \C^\times$ and consider the morphism
\begin{equation*}
i_{t_0} \colon X \to X \times \C^\times, \; x \mapsto (x,t_0).
\end{equation*}

Applying the functor $i^{-1}_{t_0}$ to the morphism \eqref{mor:prederiv}, we obtain

\begin{equation*}
\xymatrix{
	\dfrac{d\phi_t(\cdot)}{dt}\vert_{t=t_0}\colon \mu^{-1}_{t_0}\cM \ar[r]& i^{-1}_{t_0} \mu^\ast \cM \ar[r]^{i_{t_0}^{-1} \phi} & i_{t_0}^{-1} p_2^\ast\cM \ar[r]^-{i_{t_0}^{-1}\partial_t} &   i_{t_0}^{-1} p_2^\ast\cM \ar[r]^-{ev_{t_0}}& \cM.
}
\end{equation*}
In particular, when $t_0=1$, we get
\begin{equation*}
\xymatrix{
	v\colon \cM \ar[r]& i^{-1}_1 \mu^\ast \cM \ar[r]^{i_1^{-1} \phi} & i_1^{-1} p_2^\ast\cM \ar[r]^-{i_1^{-1}\partial_t} &   i_1^{-1} p_2^\ast\cM \ar[r]^-{ev_1}& \cM.
}
\end{equation*}
In other words,
\begin{align*}
v & \colon \cM \to \cM\\
& \hspace{0.4cm}s \mapsto \dfrac{d\phi_t(s)}{dt}\vert_{t=1}.
\end{align*}
The morphism $v$ is a $\C$-linear derivation of the module $\cM$.

\section{Invariant sections}

\subsection{Generalities}

We start by defining the notion of locally invariant and invariant sections. If $U$ is an open subset of $X$, we sometimes write, for brevity, $tU$ instead of $\mu(t,U)$.
\begin{definition}
	Let $(\cM, \phi) \in \Mod_{\mathrm{F}}(\cA_X^{loc})$, $U \subset X$ and $s \in \cM(U)$.
	\begin{enumerate}[(i)]
		\item The section $s$ is locally invariant at $x^\prime$  if there exists an open neighborhood $V \times U^\prime  \subset \C^\times \times X$ of $(1,x^\prime)$ such that  for every $(t,x) \in V \times U^\prime, \; \mu(t,x) \in U$ and $\phi_t(s_{\mu(t,x)})=s_x$. 
		
		\item The section $s$ is locally invariant on $U$, if it is locally invariant at every $x^\prime \in U$.
		
		\item Assume that $U \subset X$ is stable by the action of $\C^\times$. A section $s \in \cM(U)$ is invariant if for every $t \in \C^\times$, $\phi_t(s)=s$.
	\end{enumerate}
\end{definition}

\begin{lemma}\label{lem:extensionderive}
	Let $\cM \in \Mod_{\mathrm{F}}(\cA_X^{loc})$, $U \subset X$ an open subset and $s \in \cM(U)$ such that $v(s)=0$. Then, for every $x^\prime \in U$ there is a neighborhood $V \times U^\prime$ of $(1,x^\prime) \in   \C^\times \times U$  such that for every $ t^\prime \in V$
	\begin{enumerate}
		\item $\frac{d\phi_t(s\vert_{t^\prime U^\prime})}{dt}\vert_{t=t^\prime}=0$.
		
		\item  $\phi_{t^\prime}(s\vert_{t^\prime U^\prime})=s\vert_{U^\prime}$
	\end{enumerate}
\end{lemma}

\begin{proof}
	\noindent (i) Since $\mu$ is continuous, there exists a neighbourhhod $V \times U^\prime \subset U$ of $(1,x^\prime)$ such that $\mu(V \times U^\prime) \subset U$.
	Hence for every $t^\prime \in V$, we have the following equalities.
	\begin{align*}
	\frac{d\phi_{t}(s\vert_{t^\prime U^\prime})}{dt}\vert_{t=t^\prime} 
	&=\frac{1}{t^\prime}\frac{d\phi_{tt^\prime}(s\vert_{t^\prime U^\prime})}{dt}|_{t=1}\\
	&= \frac{1}{t^\prime} \phi_{t^\prime}\left(\frac{d\phi_{t}(s\vert_{t^\prime U^\prime})}{dt}|_{t=1}\right)\\
	&=\frac{1}{t^\prime}\phi_{t^\prime}(v(s\vert_{t^\prime  U^\prime}))=0.
	\end{align*}
	\noindent (ii) Shrinking $V$ and $U^\prime$ if necessary Corollary \ref{cor:loccons} implies that locally for every $t^\prime \in V$ $\phi_{t^\prime}(s\vert_{t^\prime U^\prime})=s\vert_{U^\prime}$.
\end{proof}
\begin{proposition}\label{prop:invequi}
	Let $(\cM,\phi) \in \Mod_{\mathrm{F},\coh}(\cA_X^{loc})$ and $U$ be an open subset of $X$ stable by the action of $\C^\times$ and let $s \in \cM(U)$. The following conditions are equivalent
	\begin{enumerate}[(1)]
		\item for every $t \in \C^\times$, $\phi_t(s)=s$,
		\item $v(s)=0$.
	\end{enumerate}
\end{proposition}

\begin{proof}
	It follows from Corollary \ref{cor:annulation} that $(i) \Rightarrow$ (ii).
	
	Let us prove that (ii) $\Rightarrow$ (i). It follows from Lemma \ref{lem:extensionderive} that for every  $t^\prime \in \C^\times$, $\frac{d \phi(s)}{dt}\vert_{t=t^\prime}=0$. Then, Corollary \ref{cor:loccons} implies that $\phi(s)=s$
\end{proof}

\begin{lemma}\label{lem:resinvariant}
	Let $\cM \in \Mod_{\mathrm{F},\coh}(\cA_X^{loc})$. Locally, there exists a coherent $\cA_X$-module $\cM_0 \subset \cM$ such that $\cM \simeq \cM_0^{loc}$ and $v(\cM_0) \subset \cM_0$.
\end{lemma}

\begin{proof}
	The question is local and $\cM$ is coherent. Hence we can assume that $\cM$ is finitely generated. Thus, there exist $s_1,\ldots,s_n \in \cM$ such that $\cM=\sum_{i=
		1}^n \cA_X^{loc} s_i$. This implies that there is an $l \in \N$ such for every $1 \leq i \leq n$, there exist $a_{ij} \in \hbar^{-l}\cA_{\C^\times \times X}$ such that
	\begin{equation*}
	\phi_t(s_i)=\sum_{j=1}^n a_{ij}(t) s_j.
	\end{equation*}
	Setting $t=e^u$, this implies that
	\begin{equation*}
	v^{k}(s)= \sum_{i=j}^n \frac{d^k}{du^k}a_{ij}(e^u)\bigg\rvert_{u=0} s_i.
	\end{equation*}
	Setting $\cN=\sum_{i=1}^{n} \cA_X s_i$, it follows that for every $k \in \N, v^{k}(s) \in \hbar^{-l} \cN$. We consider the submodules 
	\begin{align*}
	\cM_0=\sum_{k, \,i} \cA_X \, v^k(s_i) && \cM_{0, \leq p}=\sum_{\substack{1 \leq i \leq n \\ \; 0\leq k \leq p}} \cA_X \, v^k(s_i)
	\end{align*}
	of $\hbar^{-l} \cN$. It is clear that $\cM_0$ is stable by $v$, that $\cM_{0, \leq p}$ is a coherent $\cA_X$-module,  that for every $p \in \N, \;\cM_{0, \leq p} \subset \cM_{0, \leq p+1}$ and $\cM_0= \bigcup_{p \geq 0} \cM_{0,\leq p}$. Since $\cA_X$ is a noetherian sheaf of algebras and $\hbar^{-l} \cN$ is a coherent $\cA_X$-module, the sequence $(\cM_{0, \leq p})_{p \in \N}$ is locally stationary. Thus, there exists a covering $(U_j)_{j \in J}$ of $X$ such that for every $j \in J$ there exists $p_j \in \N$ such that  $\cM_0|_{U_j}= \cM_{0,\leq p_j}$. Hence, $\cM_0$ is coherent.
	
\end{proof}

\begin{theorem}\label{thm:invariantgeneration}
	Assume that the action $\mu$ is free and let $\cM \in \Mod_{\mathrm{F},\coh}(\cA^{loc}_X)$. Then $\cM$ is locally finitely generated by locally invariant sections.
\end{theorem}

\begin{notation}
	We introduce the following notation. Let $x \in \C^n$ and $\underline{\delta}=(\delta_1, \ldots, \delta_n)$ (resp. $\underline{\eta}=(\eta_1, \ldots, \eta_n))$ with $\delta_i > 0$ for $1 \leq i \leq n$ (resp. $\eta_i > 0$ for $1 \leq i \leq n$). We denote by $R(x,\underline{\delta},\underline{\eta})$ the subset of $\C^n$ the elements of which are the $z=(z_1,\ldots,z_n) \in \C^n$ such that for every $1 \leq i \leq n$
	\[
	\Re(x_i)-\delta_i<\Re(z_i)<\Re(x_i)+\delta_i \; \textnormal{and} \; \Im(x_i)-\eta_i<\Im(z_i)<\Im(x_i)+\eta_i.
	\]
	We call such a subset of $\C^n$ a rectangle centered in $x$.
\end{notation}

\begin{proof}[Proof of Theorem \ref{thm:invariantgeneration}]
	By lemma \ref{lem:resinvariant}, $\cM$ has a coherent $\cA_X$-lattice $\cM_0$ stable by $v$. Let us show that this lattice is generated by locally invariant sections. This problem is local. So, it is sufficient to work in an open neighborhood $U$ of a point $x \in X$. Since $\cM_0$ is coherent, we can assume it is finitely generated on $U$ by a family $(e_i)_{1 \leq i \leq l}$. For every $1 \leq i \leq l$ there exists $a_{ij} \in \cA_U$ with $1 \leq j \leq l$ such that
	\begin{equation*}
	v(e_i)=\sum_{j=i}^l a_{ij} e_j.
	\end{equation*}
	We form the matrix $A=(a_{ij})_{1 \leq i,j \leq l} \in \mathrm{M}_l(\cA_U)$ and set
	\begin{align*}
	e=\begin{pmatrix}
	e_1\\
	\vdots\\
	e_l
	\end{pmatrix}
	&& v(e)=\begin{pmatrix}
	v(e_1)\\
	\vdots\\
	v(e_l)
	\end{pmatrix}.
	\end{align*}
	It follows that $v(e)=Ae$.
	
	If $B \in \mathrm{GL}_l(\cA_U)$ is such that $v(Be)=0$ then, $Be$ will provide a generating family of $\cM_0$, formed of locally invariant sections. 
	We have the following equalities.
	\begin{align*}
	v(Be)&=v(B)e+Bv(e)\\
	&=v(B)e+BAe.
	\end{align*}
	Therefore, we are looking for $B$ in $\mathrm{GL}_l(\cA_U)$ such that  
	\begin{equation*}
	v(B)e+BAe=0.
	\end{equation*}
	Hence, it is sufficient to prove the existence of $B \in \mathrm{GL}_l(\cA_U)$ such that
	\begin{equation}\label{eq:master}
	v(B)+BA=0.
	\end{equation}
	
	Shrinking $U$ if necessary, we can further assume that $\cA_U$ is isomorphic to a star-algebra $(\cO_U^\hbar,\star)$ the star-product of which is given by 
	\begin{equation*}
	f \star g= \sum_{i \geq 0} P_i(f,g)
	\end{equation*}
	where the $P_i$ are bidifferential operators. Writting $A=\sum_{k \geq 0} A_k \hbar^k$ and $B=\sum_{j \geq 0} B_j \hbar^j$ with $A_k$ and $B_j$ in $\mathrm{M}_l(\cO_U)$, we obtain
	\begin{align*}
	BA&=(\sum_{j \geq 0}\hbar^j B_j) \star (\sum_{k \geq 0} \hbar^k A_k)\\
	&=\sum_{n \geq 0} \hbar^n (\sum_{k+j+m=n} P_m(B_j,A_k))
	\end{align*}
	and
	\begin{align*}
	v(B)&=\sum_{n \geq 0} \hbar^n \left( \sum_{i+k=n} v_k(B_i) + m \,n \, B_n \right)\\
	&=\sum_{n \geq 0} \hbar^n \left(\sum_{\substack{i+k=n \\ k\neq 0}} v_k(B_i)+v_0(B_n) + m \,n \, B_n \right).
	\end{align*}
	Thus, equation \eqref{eq:master} is equivalent to the recursive system of equations
	\begin{equation}\label{eq:mastercoor}
	\forall n \in \N, \;v_0(B_n) + B_n(mn \, \id +A_0)\,+ \sum_{\substack{i+k=n \\ i\neq n}} v_k(B_i)+\sum_{\substack{k+j+m=n \\ j \neq n}}P_m(B_j,A_k)=0.
	\end{equation} 
	Setting $C_0=0$ and $C_n=- \left( \sum_{\substack{i+k=n \\ i\neq n}} v_k(B_i)+\sum_{\substack{k+j+m=n \\ j \neq n}}P_m(B_j,A_k) \right)$, the system \eqref{eq:mastercoor} is rewritten as
	\begin{equation}\label{eq:mastercoorcomp}
	\forall n \in \N, \;v_0(B_n) + B_n(mn \, \id +A_0)=C_n.  
	\end{equation}
	Notice that $C_n$ depends only of the $B_j$ with $j<n$. 
	
	Since the action of $\C^\times$ on $X$ is free, $v_0$ does not vanish and we can find a local coordinate system $(z_1,\ldots,z_{n})$ on an open neighborhood of $x$ (that we still denote $U$) such that in this coordinate system
	\begin{equation*}
	v_0 \equiv \partiel{}{z_1}.
	\end{equation*} 
	Thus, the system \eqref{eq:mastercoorcomp} becomes
	\begin{equation}\label{eq:normalized}
	\forall n \in \N, \; \partiel{B_n}{z_1}+B_n(mn \id + A_0)=C_n. 
	\end{equation}
	
	If $(z;u)$ is the coordinnate system on $T^\ast U$ associated to the coordinate system $(z)=(z_1,\ldots,z_n)$ on $U$, it follows from \cite{SatoKas} that the characteristic variety of the system \eqref{eq:normalized} is $\lbrace u_1 = 0 \rbrace$.
	
	Since $U$ is open, there exists a rectangle $R(x,\underline{\delta},\underline{\eta}) \subset U$. Set $Y=\lbrace z \in U|z_1=x_1 \rbrace$ and $Y^\prime=Y \cap R(x,\underline{\delta},\underline{\eta})$. Consider a  function $g \in \cO_U(R(x,\underline{\delta},\underline{\eta}))$. We have the following Cauchy problem
	\begin{equation}\label{eq:sysgen}
	\begin{cases}
	\partiel{f}{z_1}+f(mn \id + A_0)=g\\
	f|_{Y^\prime}= \mathrm{I}_l.
	\end{cases}
	\end{equation}
	The hypersurface $Y^\prime$ is non-characteristic. Thus, it follows from the Cauchy-Kowaleski Theorem (see for instance \cite[Theorem 3.1.1]{Sch}) that the equation \eqref{eq:sysgen} admits a solution $f$ in a open neighborhood $\Omega$ of $Y^\prime$. Moreover, any hyperplane the normal of which is the limit of characteristic directions in at least one point of $R(x,\underline{\delta},\underline{\eta})$ intersects $Y^\prime$ since the characteristic variety of the system \eqref{eq:sysgen} is $\lbrace u_1= 0 \rbrace$. Then it follows from \cite[Theorem 2.1]{SchBo} that $f$ extends to $R(x,\underline{\delta},\underline{\eta})$. This proves that for every $n \in \N$, the equation \eqref{eq:mastercoor} has a solution $B_n$ with $B_n|_Y=\mathrm{I}_l$ and defined on $R(x,\underline{\delta},\underline{\eta})$. This implies in particular that $B_0(p)=\mathrm{I}_l$ is invertible. Hence, $B_0$ is invertible in a neighborhood $V$ of $x$. This ensures that $B=\sum_{j \geq 0} \hbar^j B_j$ is invertible on $V$ which proves the claim. 
\end{proof}

\subsection{The case of free and proper actions}
We now assume that the action of $\C^\times$ on $X$ is \textbf{free and proper}. We set $Y=X / \C^\times$ and denote by $p\colon X \to Y$ the canonical projection. The morphism $p$ is a $\C^\times$-principal bundle. We say that an open subset $V$ of $X$ is equivariant if $V= \mu(\C^\times \times V)$.

\begin{lemma}\label{lem:extension}
	Let $\cM \in \Mod_{\mathrm{F}}(\cA_X^{loc})$. 
	\begin{enumerate}[(i)]
		\item If $s \in \cM$ is a locally invariant section, then $s$ is locally the restriction of a globally invariant section.
		
		\item   If  $\cM$ is locally finitely generated by locally invariant sections, then $\cM$ is locally finitely generated by invariant sections.
	\end{enumerate}
\end{lemma}

\begin{proof}
	\noindent (i) Let $U$ be an open subset of $X$, assume that $s \in \cM(U)$ and let $x \in U$. Since $s$ is locally invariant on $U$ there exist a neigbourhoud $U^\prime \subset U$ of $x$ and a neighborhood $V$ of $1 \in \C^\times$ such that for every $t \in V, \; x^\prime \in U^\prime$, $\mu(t,x^\prime) \in U$ and $\phi_t(s)=s$.
	
	(a) Shrinking $U^\prime$ and $V$ if necessary we can assume that the pair $U^\prime$ and $V$ satisfies the following property. Let $x_0, \, x_1 \in U^\prime$ and assume that there exists $t \in \C^\times$ such that $x_1=\mu(t,x_0)$ then $t \in V$.

	(b) Let $y \in W=p^{-1}(p(U^\prime))$. There exists $(t_0,x_0) \in  \C^\times \times U^\prime$ such that $y=\mu(t_0,x_0)$. We set
	\begin{equation*}
	\tilde{s}_y=\phi^{-1}_{t_0}(s_{x_0}).
	\end{equation*}
	Let us show that the section $\tilde{s}$ is well defined. Assume that there exist $(t_1,x_1) \in \C^\times \times U^\prime$ such that $y=\mu(t_1,x_1)$. Then there exists $t \in V$ such that $x_1=\mu(t,x_0)$.
	It follows that
	\begin{equation*}
	\phi^{-1}_{t_1}(s_{x_1})= \phi^{-1}_{t_0}(\phi^{-1}_{t}(s_{x_1}))=\phi^{-1}_{t_0}(s_{x_0})=\tilde{s}_y.
	\end{equation*}
	
	(ii) Since $\cM$ is locally finitely generated by locally invariant sections, there exists an open subset $U$ of $X$ and locally invariant sections $(s_1, \ldots, s_n)$ generating $\cM \vert_U$. Keeping the notation and applying the construction of (i) to the family $(s_1, \ldots, s_n)$, we obtain some open sets $U^\prime \subset U$ and $V \subset \C^\times$ satisfying the same properties as in (i) and a family of invariant sections $(\tilde{s}_1,\ldots, \tilde{s}_n)$ defined on the open subset $W=p^{-1}(p(U^\prime))$ of $X$. It remains to prove that $(\tilde{s}_1,\ldots, \tilde{s}_n)$ is a generating family of $\cM\vert_W$. Since $W$ is stable by the action, we can assume for the sake of simplicity that $X=W$. The problem is now equivalent to check that the morphism of sheaves of $\cA_X^{loc}$-modules $u \colon ( \cA_X^{loc} )^n \to \cM, \; e_i \mapsto \tilde{s}_i$ is an epimorphism. There is the following commutative diagram of $p_1^{-1} \cO_{\C^\times}$-modules.
	\begin{equation}\label{diag:epi}
	\xymatrix{ \mu^\ast ( \cA_X^{loc} )^n  \ar[r]^-{\mu^\ast u} \ar[d]_-{\theta} &\mu^\ast\cM \ar[d]^-{\phi} \\
		p_2^\ast ( \cA_X^{loc} )^n  \ar[r]^-{p_2^\ast u}  & p_2^\ast\cM .}
	\end{equation}
	
	Let $x_1 \in X$. Then there exists $x_0 \in U^\prime$ and $t \in \C^\times$ such that $x_1=\mu(t,x_0)$. It follows from the assumption that we can further assume that $t \in V$. We consider the map
	\begin{align*}
	i_{(t,x_0)} \colon  \lbrace \pt \rbrace \to \C^\times \times X, \; \pt \mapsto (t,x_0)
	\end{align*}
	and the evaluation map
	\begin{equation*}
	i_{(t,x_0)}^\sharp \colon i_{(t,x_0)}^{-1}p_2^{-1} \cO_{\C^\times} \simeq \cO_{\C^\times,t} \to \C.
	\end{equation*}
	These two maps allow us to define the morphism of ringed spaces
	\begin{equation*}
	(i_{(t,x_0)},i_{(t,x_0)}^\sharp) \colon(\lbrace \pt \rbrace, \C) \to ( \C^\times \times X, p_2^{-1} \cO_{C^\times}) 
	\end{equation*}
	Applying the functor $i_{(t,x_0)}^\ast$ to the diagram \eqref{diag:epi}, we obtain
	\begin{equation*}
	\xymatrix{  ( \cA_X^{loc} )_{x_1}^n  \ar[r]^-{u_{x_1}} \ar[d]_-{\theta_{(x_0,t)}} &\cM_{x_1} \ar[d]^-{\phi_{(x_0,t)}} \\
		( \cA_X^{loc} )_{x_0}^n  \ar[r]^-{u_{x_0}}  &\cM_{x_0} 
	}
	\end{equation*}
	where the two vertical arrows are isomorphisms and the map $u_{x_0}$ is an epimorphism. This implies that $u_{x_1}$ is an epimorphism which proves the claim.
\end{proof}

\begin{corollary}\label{cor:geninvglob}
	Let $\cM \in \Mod_{\mathrm{F},\coh}(\cA^{loc}_X)$. Then there exists a covering $(V_i)_{i \in I}$ of $X$ by equivariant open subsets of $X$ such that for each $i \in I$, $\cM \vert_{V_i}$ admits, in $\Mod_{\mathrm{F}}(\cA_X^{loc}\vert_{V_i})$, a presentation of length one by free modules of finite rank.
\end{corollary}

\begin{proof}
	This follows from Theorem \ref{thm:invariantgeneration} and Lemma \ref{lem:extension} (ii).
\end{proof}

\section{From DQ-modules to modules over the ring of invariant sections}

Recall that we assume that the action of $\C^\times$ on $X$ is \textbf{free and proper}.

\subsection{Equivariant extension and invariant sections functors}

We define the sheaf of locally invariant sections of $\cA_X$ as the sheaf on $X$ such that for every open set $U$
\begin{equation*}
\cA^{\C^\times}_X(U) = \{ s \in \cA_X(U) | v(s)=0 \}
\end{equation*}
and we also set
\begin{equation*}
\cA^{loc,\,\C^\times}_X(U) = \{ s \in \cA^{loc}_X(U) | v(s)=0 \}.
\end{equation*}
The sheaf of invariant sections of $\cA_X$ is defined as the subsheaf $\cB_Y(0)$ of $p_\ast \cA_X$ which is given, for every open set $V \subset Y$, by
\begin{equation*}
\cB_Y(0)(V) = \{ s \in p_\ast\cA_X(V) | v(s)=0 \}=\{ s \in p_\ast\cA_X(V) | \forall t \in \C^\times,  \theta_t(s)=s \}.
\end{equation*}
This is a sheaf of $\C$-algebra. We define $\cB_Y$ similarly i.e.
\begin{equation*}
\cB_Y(V) = \{ s \in p_\ast\cA^{loc}_X(V) | v(s)=0 \}=\{ s \in p_\ast\cA^{loc}_X(V) | \forall t \in \C^\times, \theta_t(s)=s \}.
\end{equation*}
By definition of $\cB_Y(0)$ and $\cB_Y$, there are morphisms of algebras
\begin{equation}\label{map:Binv}
p^{-1}\cB_Y(0) \to  \cA^{\C^\times}_X,  
\end{equation}
\begin{equation}\label{map:Binvloc}
p^{-1}\cB_Y \to \cA^{loc, \, \C^\times}_X.
\end{equation}

\begin{lemma}
	The morphisms \eqref{map:Binv} and \eqref{map:Binvloc} are isomorphims of $\C$-algebras.
\end{lemma}

\begin{proof}
	This follows from Lemma \ref{lem:extension}.
\end{proof}

We define the functor of locally invariant sections as follows
\begin{align*}
(\cdot)^{\C^\times}:&\Mod_\mathrm{F}(\cA^{loc}_X) \to \Mod(\cA_X^{loc,\C^\times})\\ &\hspace{1.45cm}\cM \mapsto \cM^{\C^\times}
\end{align*}
where $\cM^{\C^\times}$ is the subsheaf of $\cM$ such that for every open $U \subset X$
\begin{equation*}
\cM^{\C^\times}(U)=\lbrace s \in \cM(U) | v(s)=0 \rbrace
\end{equation*}
The functor of globally invariant sections is defined by
\begin{align*}
p_\ast^{\C^\times}:&\Mod_\mathrm{F}(\cA^{loc}_X) \to \Mod(\cB_Y)\\ &\hspace{1.45cm}\cM \mapsto \pinv \cM := p_\ast(\cM^{\C^\times}).
\end{align*}
Note that by definition of $p_\ast^{\C^\times}$, there is a natural transformation $i \colon p_\ast^{\C^\times} \to p_\ast$, such that for $(\cM,\psi) \in \Mod_{\mathrm{F}}(\cA^{loc}_X)$
\begin{equation}\label{nattr:i}
i_\cM \colon p^{\C^\times}_\ast \cM \hookrightarrow p_\ast \cM
\end{equation}
is given by the inclusion.

Consider the subsheaf $\widetilde{\cM}$ of $p_\ast \cM$ defined by
\begin{equation*}
\widetilde{\cM}(V):=\lbrace s \in p_\ast \cM(V) \, \vert \, \forall t \in \C^\times, \phi_t(s)=s \rbrace
\end{equation*}
By definition of $p_\ast^{\C^\times}\cM$ and $\widetilde{\cM}$, there is a canonical map
\begin{equation}\label{mor:equivinvsheaf}
\widetilde{\cM} \to p_\ast^{\C^\times}\cM.
\end{equation}
This is an isomorphism by Proposition \ref{prop:invequi}. Hence, we do not make any distinction between $p_\ast^{\C^\times}\cM$ and $\widetilde{\cM}$ and denote both of them by $p_\ast^{\C^\times}\cM$.

By definition of $\cB_Y$ there is a morphism of sheaves of algebras
\begin{equation}\label{mor:extensioninv}
p^{-1}\cB_Y \to \cA_X^{loc}.
\end{equation}
As $p^{-1}\cB_Y\simeq \cA_X^{loc, \C^\times}$, the morphism \eqref{mor:extensioninv} is a monomorphism.
The morphism \eqref{mor:extensioninv} allows us to define the functor
\begin{align*}
p^\ast \colon \Mod(\cB_Y) \to \Mod(\cA_X^{loc}) && \cN \mapsto p^\ast \cN= \cA^{loc}_X \otimes_{p^{-1}\cB_Y} p^{-1} \cN.
\end{align*}
There is the following adjunction

\begin{equation*}
\xymatrix{
	p^\ast \colon \Mod(\cB_Y) \ar@<.4ex>[r]&  \ar@<.4ex>[l]\Mod(\cA^{loc}_X) \colon p_\ast.
}
\end{equation*}

The module $p^\ast \cN$ is canonically equipped with an F-action. Since $p \circ \mu = p \circ p_2$, there is an isomorphism of algebras $\mu^{-1}p^{-1} \cB_Y \simeq p_2^{-1}p^{-1} \cB_Y$. Using the morphism \eqref{mor:extensioninv}, we obtain the following commutative diagram
\begin{equation*}
\xymatrix{
	\mu^{-1}p^{-1} \cB_Y \ar[r] \ar[d]_-{\wr} & \mu^{-1}\cA_X^{loc} \ar[r]^-{\theta} & \cA^{loc}_{\C^\times \times X} \ar@{=}[d]\\
	p_2^{-1}p^{-1} \cB_Y \ar[r] & p_2^{-1}\cA_X^{loc} \ar[r] & \cA^{loc}_{\C^\times \times X}.
}
\end{equation*}
This implies that $\mu^\ast p^\ast \simeq p_2^\ast p^\ast$. Hence, for $\cN \in \Mod(\cB_Y)$, there is an isomorphism $\phi \colon \mu^\ast p^\ast \cN \stackrel{\sim}{\to} p_2^\ast p^\ast \cN$.
That is $(p^\ast\cN, \phi) \in \Mod_{\mathrm{F}}(\cA_X^{loc})$. We obtain a functor $p_{\C^\times}^\ast$
\begin{equation*}
p^\ast_{\C^\times} \colon \Mod(\cB_Y) \to \Mod_{\mathrm{F}}(\cA_X^{loc}), \quad \cN \mapsto p^\ast_{\C^\times} \cN \colon \!\!\!= (p^\ast \cN, \, \phi).
\end{equation*}

\begin{proposition}\label{prop:adjunction}
	The functors
	\begin{equation*}
	\xymatrix{
		p^\ast_{\C^\times} \colon \Mod(\cB_Y) \ar@<.4ex>[r]&  \ar@<.4ex>[l]\Mod_{\mathrm{F}}(\cA^{loc}_X) \colon p_\ast^{\C^\times}
	}
	\end{equation*}
	form the adjoint pair $(p^\ast_{\C^\times},\,p_\ast^{\C^\times})$.
\end{proposition}

\begin{proof}
	\noindent (i) Let $\cN \in \Mod(\cB_Y)$. We start by constructing the unit of the adjunction $(p^\ast_{\C^\times},\,p_\ast^{\C^\times})$. Consider the unit $\eta^\prime \colon \id \to p_\ast p^\ast $ of the adjunction  
	\begin{equation*}
	\xymatrix{
		p^\ast \colon \Mod(\cB_Y) \ar@<.4ex>[r]&  \ar@<.4ex>[l]\Mod(\cA^{loc}_X) \colon p_\ast.
	}
	\end{equation*}
	The stalk of the map $\eta^\prime_\cN$ at $y \in Y$ is given by
	\begin{equation*}
	\cN_y \to \varinjlim_{y \in V} \cA_X^{loc}(p^{-1}(V)) \te_{\cB_{Y,y}} \cN_y, \; n \mapsto 1 \otimes n.
	\end{equation*}
	Sections of the form $1 \otimes n$ are invariant section of $p_{\C^\times}^\ast\cN$. Thus, the preceding map factorizes through $p^{\C^\times}_\ast p_{\C^\times}^\ast \cN$ and we obtain the following commutative diagram.
	\begin{equation}\label{diag:comunit}
	\xymatrix{
		\cN \ar[rr]^{\eta_\cN^\prime} \ar[rd]_-{\eta_\cN} & &  p_\ast p^\ast \cN\\
		& p^{\C^\times}_\ast p_{\C^\times}^\ast \cN \ar[ru]_-{i_{p^\ast_{\C^\times}\!\cN}} &
	}
	\end{equation}
	We have obtained a natural transformation $\eta: \id \to p_\ast^{\C^\times}p_{\C^\times}^\ast$.

	Let $(\cM,\psi) \in \Mod_{\mathrm{F}}(\cA^{loc}_X)$. Let $\varepsilon^\prime \colon p^\ast p_\ast \to \id$ be the counit of the adjunction
	\begin{equation*}
	\xymatrix{
		p^\ast \colon \Mod(\cB_Y) \ar@<.4ex>[r]&  \ar@<.4ex>[l]\Mod(\cA^{loc}_X) \colon p_\ast.
	}
	\end{equation*}
	We define the following natural transformation
	\begin{equation}\label{map:counit}
	\varepsilon = \varepsilon^\prime \circ p^\ast i
	\end{equation}
	where $i$ is the natural transformation defined in \eqref{nattr:i}.
	By construction, the following diagram commutes.
	\begin{equation}\label{diag:comcounit}
	\xymatrix{
		p^\ast p_\ast \cM \ar[rr]^-{\varepsilon^\prime_\cM} & & \cM \\
		&p_{\C^\times}^\ast p_\ast^{\C^\times} \cM \ar[ru]_-{\varepsilon_\cM} \ar[lu]^-{p^\ast i_{\cM}} &
	}
	\end{equation}
	Let $x$ in $X$. The stalk at $x$ of morphism \eqref{map:counit} is given by
	\begin{equation*}
	\cA^{loc}_{X,x} \te_{\cB_{Y, p(x)}} (p_\ast^{\C^\times} \cM)_{p(x)} \to \cM_x, \; a \otimes m \mapsto am
	\end{equation*}
	This implies that the following diagram
	\begin{equation*}
	\xymatrix{
		\mu^\ast p^\ast p_\ast^{\C^\times}  \cM \ar[d]_-{\phi} \ar[r]^-{\mu^\ast \varepsilon_\cM} &\mu^\ast \cM \ar[d]^{\psi} \\
		p_2^\ast p^\ast p_\ast^{\C^\times}  \cM \ar[r]^-{p_2^\ast \varepsilon_\cM}& p_2^\ast \cM
	}
	\end{equation*}
	where $\phi$ is provided by the functor $p_{\C^\times}^\ast$ is commutative. Hence, morphism \eqref{map:counit} is equivariant i.e. $\varepsilon$ is a morphism of $\Mod_{\mathrm{F}}(\cA_X^{loc})$.
	
	In view of diagrams \eqref{diag:comunit} and \eqref{diag:comcounit}, the following diagrams are commutative
	\begin{equation*}
	\xymatrix{
		p^\ast \cN \ar[r]^-{p^\ast \eta^\prime_\cN} \ar@/^2pc/[rr]^-{\id} 
		& p^\ast p_\ast p^\ast \cN \ar[r]^-{\varepsilon^\prime_{p^\ast\cN}}
		& p^\ast \cN\\
		p_{\C^\times}^\ast \cN \ar[r]_-{p^\ast_{\C^\times} \eta_\cN} \ar@{=}[u] 
		& p_{\C^\times}^\ast p_\ast^{\C^\times} p_{\C^\times}^\ast \cN \ar[r]_-{\varepsilon_{p^\ast_{\C^\times}\cN}} \ar[u]^-{p^\ast i_{p^\ast_{\C^\times}\!\!\cN}}
		& p_{\C^\times}^\ast \cN  \ar@{=}[u]
	}
	\end{equation*}

	\begin{equation*}
	\xymatrix{
		p_\ast \cM \ar[r]^-{ \eta^\prime_{p_\ast\cM} } \ar@/^2pc/[rr]^-{\id} 
		& p_\ast p^\ast p_\ast \cM \ar[r]^-{p_\ast \varepsilon^\prime_\cM}
		& p_\ast \cM\\
		&p_\ast p_{\C^\times}^\ast p_\ast^{\C^\times} \cM \ar[u]^-{p_\ast p^\ast i_\cM} \ar[ru]^-{p_\ast \varepsilon_\cM}&\\
		p^{\C^\times}_\ast \cM \ar[r]_-{\eta_{p_\ast^{\C^\times}\cM}} \ar@{^{(}->}[uu]^-{i_\cM} \ar[ru]^-{\eta^\prime_{p_\ast^{\C^\times}\cM}}
		& p^{\C^\times}_\ast p^\ast_{\C^\times} p^{\C^\times}_\ast \cM \ar[r]_-{p_\ast^{\C^\times} \varepsilon_\cM} \ar[u]_-{i_{p^\ast_{\C^\times} p^{\C^\times}_\ast\cM}}
		& p^{\C^\times}_\ast \cM  \ar@{^{(}->}[uu]_-{i_\cM}
	}
	\end{equation*}
which proves that $(p^\ast_{\C^\times}, p_\ast^{\C^\times})$ is an adjoint pair.\\
\end{proof}

\subsection{Coherence of the sheaf of invariant sections}

In this subsection, we prove the coherence of $\cB_Y$. The following result is elementary.

\begin{lemma}\label{lem:derivsol}
	Let $x \in X$. Let $v_0$ be an holomorphic derivation of $\cO_X$ that does not  vanish at $x$. Then there exists a neighborhood $U$ of $x$ such that for every $c \in \C$ the map
	\begin{equation*}
	v_0 + c \colon \cO_X(U) \to \cO_X(U)\, ,\quad f\mapsto v_0(f)+cf
	\end{equation*}
	is surjective. 
\end{lemma}

\begin{lemma}\label{lem:hfactor}
	Let $s \in \cA_X(X)$ and assume that there is a covering $(U_i)_{i \in I}$ of $X$ and $u_i \in \cA_{X}(U_i)$ such that $s|_{U_i}=\hbar u_i$. Then, there exists $u \in \cA_X(X)$ such that $s=\hbar u$  
\end{lemma}

\begin{proof}
	On $U_i \cap U_j$, $\hbar u_i \vert_{U_{ij}}=\hbar u_j \vert_{U_{ij}}$. As $\hbar$ is not a zero divisor, $u_i \vert_{U_{ij}}= u_j \vert_{U_{ij}}$.
\end{proof}

\begin{lemma}\label{lem:epiinv1}
	Let $x \in X$. There exist an equivariant neighborhood $V$ of $x$ and a section $s \in \cA_X(V)$ such that $v(s)=0$ and $s=\hbar u$ with $u \in \cA_X(V)$ an invertible section.
\end{lemma}

\begin{proof}
	Let $x \in X$ and $m \in \Z$. Since $v_0$ does not vanish in $x$, there exists an open neighborhood $U$ of $x$ and an hypersurface $S$ in $U$ which is non-characteristic for $v_0+m$ on $U$. Hence, the following Cauchy problem
	\begin{equation*}
	\begin{cases}
	v_0(f)+mf=0\\
	f \vert_S=1
	\end{cases}
	\end{equation*}
	has a solution $f_0$ in an open neighborhood $U^\prime$ of $x$. Shrinking $U^\prime$ if necessary we can further assume that $f_0$ is invertible on it.
	
	We now show that there exists $u=\sum_{i \geq 0} \hbar^i u_i \in \cA_{X,x}$ such that
	\begin{align*}
	u_0= f_0  \quad  \textnormal{and} \quad v(\hbar u)=0. 
	\end{align*}
	It follows from equation \eqref{eq:basederivation} that this is equivalent to show that there exists $u=\sum_{i \geq 0} \hbar^i u_i \in \cA_{X,x}$ such that for every $n \in \N$
	\begin{equation*}
	\sum_{i+k=n} v_k(u_i)+ m \,(n+1) \,u_n=0.
	\end{equation*}
	Thus, it remains to show that the recursive system
	\begin{equation*}
	\begin{cases}
	v_0(u_n)+ m \,(n+1) \,u_n=\sum_{\substack{i+k=n \\ k\neq 0}}v_k(u_i)\\
	u_0=f_0
	\end{cases}
	\end{equation*}
	has a solution $u$ in an open neighborhood $V^\prime$ of $x$. This follows from Lemma \ref{lem:derivsol}. Finally using Lemma \ref{lem:extension}, we extend the section $\hbar u$ to an invariant section $s$ defined on an equivariant open set $V$.

	Let us check that there exists an invertible section $\widetilde{u} \in \cA_X(V)$ such that $s=\hbar \widetilde{u}$.
	Let $y \in V$. By construction (see the proof of Lemma \ref{lem:epiinv1}), $s_y=\phi_{t_0}^{-1}(\hbar u_{x_0})= \hbar \, t_0^{-m} \phi_{t_0}^{-1}(u_{x_0})$ and $t_0^{-m} \phi_{t_0}^{-1}(u_{x_0})$ is an invertible section. Finally, Lemma \ref{lem:hfactor} implies the existence of $\widetilde{u}$.
\end{proof}

Assume that there exists a section $s \in \cA_X(X)$ such that $v(s)=0$ and $s=\hbar u$ with u invertible. Consider the exact sequence
\begin{equation}\label{diag:seqcom}
0 \to s^n \cA_X \to \cA_X \stackrel{\pi}{\to} \cA_X / s^n \cA_X \to 0
\end{equation}
and apply the functor $(\cdot)^{\C^\times}$ to it. We obtain the exact sequence 
\begin{equation}\label{diag:seqcompreinv}
0 \to (s^n \cA_X)^{\C^\times} \to \cA_X^{\C^\times} \to (\cA_X / s^n \cA_X)^{\C^\times}.
\end{equation}

\begin{lemma}
	The left exact sequence \eqref{diag:seqcompreinv} is exact.
\end{lemma}
\begin{proof}
	The proof is similar to the one of Lemma \ref{lem:epiinv1}.
\end{proof}

\begin{lemma}\label{lem:quotientcommutation}
	Assume that there exists a section $s \in \cA_X(X)$ such that $v(s)=0$ and $s=\hbar u$ with u invertible. Then
	\begin{equation*}
	p_\ast^{\C^\times}(\cA_X / s^n \cA_X)\simeq \cB_Y(0)/s^n \cB_Y(0).
	\end{equation*}
\end{lemma}

\begin{proof}
	Notice that $p_\ast s^n \cA_X \simeq s^n p_\ast \cA_X$. Hence, $p^{\C^\times}_\ast( s^n \cA_X) \simeq s^n \cB_Y(0)$.
	
	Applying $p_\ast^{\C^\times}$ to the sequence \eqref{diag:seqcom}, we obtain the following left exact sequence
	\begin{equation}\label{diag:seqcominv}
	0 \to s^n \cB_Y(0) \to \cB_Y(0) \stackrel{p_\ast^{\C^\times}\pi}{\longrightarrow}  p_\ast^{\C^\times}(\cA_X / s^n  \cA_X).
	\end{equation}
	Let us show that $p_\ast^{\C^\times}\pi$ is surjective. Let $y \in Y$, $a \in p_\ast^{\C^\times}(\cA_X / s^n  \cA_X)_y$. Thus, there exists an open subset $V \subset Y$ such that $y \in V$, $a \in \cA_X / s^n  \cA_X(U)$ where $U=p^{-1}(V)$ and $a$ is an invariant section. Choose $x \in U$ such that $y=p(x)$. Since the sequence \eqref{diag:seqcompreinv} is exact, there exists an open set $W \subset U$ containing $x$ and a locally invariant section $u \in \cA_X(W)$ such that $\pi(u)=a\vert_W$. Hence, there exists an equivariant open subset $U^\prime$ and an invariant section $b \in \cA_X(U^\prime)$ such that on a neighborhood $W^\prime \subset W \cap U^\prime$ of $x$, $b\vert_{W^\prime}=u\vert_{W^\prime}$. Moreover, shrinking $U^\prime$ if necessary, we can assume that the orbit of $W^\prime$ under the action of $\C^\times$ is $U^\prime$. As $a$ and $b$ are invariant sections, it follows that $\pi(b)=a\vert_{U^\prime}$. This proves that $p^\ast_{\C^\times}\pi$ is an epimorphism of sheaves.
\end{proof}

The following theorem is a minor variation of \cite[Theorem 1.2.5 (i)]{KS3} and the proof is the same.

\begin{theorem}\label{thm:notherianite}
	Let $k$ be a field of characteristic zero, $\cR$ a sheaf of $k$-algebras and $s$ a section of $\cR$ such that
	\begin{enumerate}[(i)]
		\item $s \cR = \cR s$,
		\item $\cR \stackrel{ \cdot s}{\longrightarrow} \cR$ is a monomorphism,
		\item $\cR \simeq \underset{n \in \N}{\varprojlim} \cR / \cR s^n $
		\item $\cR / \cR s$ is a left Noetherian ring.
	\end{enumerate}
	Then $\cR$ is a left Noetherian $k$-algebra.
\end{theorem}

\begin{theorem}\label{thm:cohinv}
	\begin{enumerate}[(i)]
		\item The sheaf $\cB_Y(0)$ is a Noetherian sheaf of $\C$-algebras.
		
		\item The sheaf $\cB_Y$ is a Noetherian sheaf of $\C$-algebras.
		
		\item The sheaf $\cA_X^{\C^\times}$ is a Noetherian sheaf of $\C$-algebras.
		
		\item The sheaf $\cA_X^{loc, \C^\times}$ is a Noetherian sheaf of $\C$-algebras.
	\end{enumerate}
\end{theorem}

\begin{proof}
	\noindent (i) We apply Theorem \ref{thm:notherianite}. Coherency is a local property. Hence, using Lemma \ref{lem:epiinv1}, we can assume that that there exists a section $s \in \cA_X(X)$ such that $v(s)=0$ and $s= \hbar u$ with $u$ invertible. Then point (i) and (ii) of Theorem \ref{thm:notherianite} are immediately satisfied. Since $p_\ast^{\C^\times}$ is a right adjoint it commutes with limits. Thus,
	\begin{equation*}
	\cB_Y(0)=p_\ast^{\C^\times} \cA_X=p_\ast^{\C^\times} (\varprojlim_n \cA_X / s^n \cA_X)\simeq \varprojlim_n p_\ast^{\C^\times} (\cA_X / s^n \cA_X).
	\end{equation*}
	Moreover, Lemma \ref{lem:quotientcommutation} implies that $p_\ast^{\C^\times} (\cA_X / s^n \cA_X) \simeq \cB_Y(0)/s^n \cB_Y(0)$. Hence,
	\begin{equation*}
	\cB_Y(0) \simeq \varprojlim_n \cB_Y(0)/s^n \cB_Y(0).
	\end{equation*}
	This proves that condition (iii) of Theorem \ref{thm:notherianite} holds as well as condition (iv) since $\cB_Y(0)/s \cB_Y(0) \simeq \cO_Y$. Thus, $\cB_Y(0)$ is Noetherian sheaf of $\C$-algebras.\\
	
	\noindent (ii) We keep the notation of (i) and consider a section $s=\hbar \, u$ as above. Consider the free algebra $\cB_Y(0)\langle T \rangle $ and impose the relations
	\begin{equation*}
	\forall a \in \cB_Y(0), \; T \cdot a = \psi_u(a) \cdot T
	\end{equation*}
	where $\psi_u(a)=u^{-1} a u$. We obtain the skew polynomial algebra $\cB_Y(0)[T; \psi_u]$. It follows from (i) and \cite[Theorem 5.1.1]{DaKas} that the ring $\cB_Y(0)[T; \psi_u]$ is Noetherian. Using \cite[Proposition A.10]{KDmod}, this implies that $\cB_Y \simeq \cB_Y(0)[T; \psi_u] / (Ts-1)\cB_Y(0)[T; \psi_u]$ is also Noetherian.\\
	
	\noindent (iii) \& (iv) We have that $\cA_X^{\C^\times} \simeq p^{-1} \cB_Y(0)$ (resp. $\cA_X^{loc, \C^\times} \simeq p^{-1} \cB_Y$). Hence the result follows from (i) (resp. (ii)) and \cite[Proposition A.14]{KDmod}.
\end{proof}

\begin{proposition}
	\begin{enumerate}[(i)]
		\item The sheaf $\cA_X$ is faithfully flat over $\cA_X^{\C^\times}$,
		\item the sheaf $\cA^{loc}_X$ is faithfully flat over $\cA_X^{loc,\C^\times}$.
	\end{enumerate}
\end{proposition}

\begin{proof}
	(i) The proof is similar to point (i) of the proof of \cite[Lemma 6.1.2 (a)]{KS3}.\\
	\noindent
	(ii) This follows from the isomorphism
	\begin{equation*}
	\cA_X^{loc} \simeq \cA_X \otimes_{\cA_X^{\C^\times}} \cA_X^{loc, \C^\times}.
	\end{equation*}
\end{proof}

\subsection{The equivalence of categories}

The aim of this subsection is to prove the following theorem.

\begin{theorem}\label{thm:equimicroDQ}
	The adjoint pair $(p^\ast_{\C^\times},p_\ast^{\C^\times})$ induces a well defined adjunction
	\begin{equation}\label{mor:adjthm}
	\xymatrix{
		p^\ast_{\C^\times} \colon \Mod_\coh(\cB_Y) \ar@<.4ex>[r]&  \ar@<.4ex>[l]\Mod_{\mathrm{F},\,\coh}(\cA^{loc}_X) \colon p_\ast^{\C^\times}.
	}
	\end{equation}
	These functors are equivalence of categories inverse to each others.
\end{theorem}

For that purpose, we  have to prove that the adjunction is well defined and that the unit and the counit of this adjunction are isomorphisms.

\begin{lemma}\label{lem:equi1}
	Let $\cM \in \Mod_{\mathrm{F}, \coh}(\cA^{loc}_X)$. The natural morphism \eqref{map:counit}
	\begin{equation*}
	\varepsilon\colon p^\ast_{\C^\times}p_\ast^{\C^\times} \cM \longrightarrow \cM 
	\end{equation*}
	is an isomorphism.
\end{lemma}

\begin{proof}
	It follows from Corollary \ref{cor:geninvglob} that the morphism \eqref{map:counit} is an epimorphism.

	Let us prove that it is a monomorphism. Let $x \in X$ and let $m=\sum_{i=1}^n w_i \otimes m_i \in \cA^{loc}_{X,x} \te_{\cA^{loc, \C^\times}_{X,x}} p_\ast^{\C^\times}(\cM)_{p(x) }$. We can assume that $m_1,\ldots,m_n$ are invariant sections of $\cM$ defined on an equivariant open subset $V$ of $X$ containing $x$. Consider the map
	\begin{equation*}
	\phi\colon(\cA^{loc}_X)^n|_V \to \cM|_V, \quad \phi(e_j)=m_j, \; 1\leq j \leq n
	\end{equation*}
	where $(e_i)_{1\leq i \leq n}$ is the canonical basis of $(\cA^{loc}_X)^n$. Then we get the following exact sequence
	\begin{equation*}
	0 \to \ker \phi \to (\cA^{loc}_X)^n|_V \stackrel{\phi}{\to} \cM|_V.
	\end{equation*}  
	The module $\ker \phi$ belongs to $\Mod_{\mathrm{F}, \, \coh}(\cA_X^{loc} \vert_V)$. It is locally finitely generated by invariant sections. Hence, there exists an equivariant open set $V^\prime$ containing $x$ and invariant sections $s_1,\ldots,s_q$ where $s_i=(s_{i1}, \ldots,s_{in})$ with $s_{ij} \in \cA_X^{loc, \C^\times}(V^\prime)$ and generating $\ker \phi \vert_{V^\prime}$. 
	
	Assume that $\varepsilon(m)=\sum_{j=1}^n w_j m_j=0$. Thus, $w=(w_1,\ldots,w_n) \in \ker \phi_x$ and we have
	\begin{align*}
	w=\sum_{i=1}^q \alpha_i s_i
	\end{align*}
	with $\alpha_i \in \cA_{X,x}^{loc}$. Then
	\begin{align*}
	\sum_{j=1}^n w_j \otimes m_j&=\sum_{j=1}^n (\sum_{i=1}^q \alpha_i s_{ij})\otimes m_j\\
	&=\sum_{i=1}^q (\alpha_i \otimes (\underbrace{\sum_{j=1}^n s_{ij} m_j)}_{=0})=0.
	\end{align*} 
	This proves that the map \eqref{map:counit} is a monomorphism.
\end{proof}

\begin{proposition}\label{prop:exactness}
	The functor $p_\ast^{\C^\times}$ restricted to $\Mod_{\mathrm{F}, \, \coh}(\cA_X^{loc})$ is exact.
\end{proposition}

\begin{proof}
	As $p_\ast^{\C^\times}$ is a right adjoint it is left exact. Let us show it is right exact on $\Mod_{\mathrm{F}, \, \coh}(\cA_X^{loc})$. Consider the exact sequence
	\begin{equation*}
	0 \to \cM^\prime \to \cM \to \cM^{\prime \prime} \to 0.
	\end{equation*}
	and apply the functor $p^\ast_{\C^\times} p_\ast^{\C^\times}$ to the above exact short exact sequence. Since $p^\ast_{\C^\times} p_\ast^{\C^\times}$ is isomorphic to the identity functor on $\Mod_{\mathrm{F}, \, \coh}(\cA_X^{loc})$, we obtain the short exact sequence
	\begin{equation*}
	0 \to p^\ast_{\C^\times} p_\ast^{\C^\times}\cM^\prime \to p^\ast_{\C^\times} p_\ast^{\C^\times}\cM \to p^\ast_{\C^\times} p_\ast^{\C^\times}\cM^{\prime \prime} \to 0.
	\end{equation*}
	The ring $\cA^{loc}_X$ is faithfully flat over $\cA^{loc, \C^\times}_X$, this implies that the sequence
	\begin{equation}\label{seq:exactp}
	0 \to p^{-1} p_\ast^{\C^\times}\cM^\prime \to  p^{-1} p_\ast^{\C^\times}\cM \to p^{-1} p_\ast^{\C^\times} \cM^{\prime \prime}\to 0
	\end{equation}
	is exact. Moreover, $p \colon X \to Y$ is surjective. Hence taking the stalks of the short exact sequence \eqref{seq:exactp} in every $x \in X$, we find  that for every $y \in Y$ the sequence
	\begin{equation*}
	0 \to (p_\ast^{\C^\times}\cM^\prime)_y \to  (p_\ast^{\C^\times}\cM)_y \to (p_\ast^{\C^\times} \cM^{\prime \prime})_y\to 0
	\end{equation*}
	is exact. This proves that $p_\ast^{\C^\times}$ is exact.
\end{proof}

\begin{lemma}\label{lem:conservationcoherent}
	The functors $p^\ast_{\C^\times}$ and $p_\ast^{\C^\times}$ preserve coherent modules.
\end{lemma}

\begin{proof}
	\noindent (i) Since $\cA_X^{loc}$ is coherent, a $\cA_X^{loc}$-modules is coherent if and only if it locally admits a presentation of length one by finitely generated free modules. This implies that $p_{\C^\times}^\ast$ preserves coherent modules.
	
	\noindent(ii) Let $\cM$ be a coherent $\cA^{loc}_X$-module endowed with an $\mathrm{F}$-action. It follows from Corollary \ref{cor:geninvglob} that there exists an equivariant open subset $V$ of $X$ such that $\cM \vert_V$ has a presentation of length one by free modules of finite rank in $\Mod_{\mathrm{F}, \, \coh}(\cA_X^{loc})$, i.e
	\begin{equation*}
	(\cA_X^{loc} \vert_V)^n \to (\cA_X^{loc} \vert_V)^m \to\cM\vert_V \to 0.
	\end{equation*}
	Applying the exact functor $p_\ast^{\C^\times}$ to the above sequence, we obtain the exact sequence 
	\begin{equation*}
	(\cB_Y \vert_{p(V)})^n \to (\cB_Y \vert_{p(V)})^m \to (p_\ast^{\C^\times} \cM) \vert_{p(V)} \to 0.
	\end{equation*}
	As $\cB_Y$ is coherent, this implies that $p_\ast^{\C^\times} \cM$ is a coherent $\cB_Y$-module.
\end{proof}

We are now ready to prove the main result of this section.

\begin{proof}[Proof of Theorem \ref{thm:equimicroDQ}]
	It follows from Proposition \ref{prop:adjunction} and  Lemma \ref{lem:conservationcoherent} that the adjunction \eqref{mor:adjthm} is well defined. Because of Lemma \ref{lem:equi1}, it only remains to show that for every $\cN \in \Mod_{\coh}(\cB_Y)$
	\begin{equation}\label{eq:couniteq}
	\eta_\cN \colon \cN \to p_\ast^{\C^\times}p_{\C^\times}^\ast \cN
	\end{equation}
	is an isomorphism. The $\cB_Y$-module $\cN$ is coherent. Hence, there is a covering $(V_i)_{i \in I}$ of $Y$ such that for each $i \in I$, $\cN \vert_{V_i}$  has a free presentations of length one. It follows that the morphism \eqref{eq:couniteq} is an isomorphism since $p^\ast_{\C^\times}$ and $p_\ast^{\C^\times}$ are right exact functors and $\eta_{\cB_Y}$ is an isomorphism.
\end{proof}

\subsection{An example: the case of $\chW_X$}

In this section, we sketch the construction of the canonical weight one $\mathrm{F}$-action on $\chW_X$. Let $M$ be a complex manifold and let X be a conical open subset of $T^\ast M$. We denote by $\chW_{X}(0)$, the restriction to $X$ of the standard quantization of the cotangent bundle and $\chW_{X} := \C^{\hbar,loc} \otimes_{\C^\hbar} \cW_{X}(0)$ (see for instance \cite[p.133]{KS3}). We write $\chE_{X}$ for the restriction to $X$ of the ring of formal microdifferentials operators on $T^\ast M$, $\chE_{X}(0)$ for the subsheaf of order zero microdifferential operators and set $\cA_{\C^\times  \times X}: = \cO^\hbar_{\C^\times} \ubtimes \cW_X(0)$.

We endow $\C$ with a coordinate $t$, $T^\ast \C$ with the coordinate system $(t,\tau)$. Similarly, we equip $\C^\times$ with the coordinate $r$ and $T^\ast \C^\times$ with the coordinate system $(r, \lambda)$. We consider the map

\begin{equation*}
\gamma \colon T^\ast(\C^\times \times \C) \to T^\ast (\C^\times \times \C), \; (r,t;\lambda,\tau) \mapsto (r,t;\lambda,\tau / r).
\end{equation*}
and the isomorphism of sheaves $\gamma^\sharp$ given, on every open subset $V$ of $T^\ast_{\tau \neq 0} (\C^\times \times \C)$,  by
\begin{align}\label{mor:preaction}
\begin{split}
\gamma_V^\sharp \colon \left( \chE_{T^\ast_{\tau \neq 0} (\C^\times \times \C), \, \hat{t}, \hat{\partial}_r}(0) \right)(V) &\stackrel{\sim}{\to} \gamma_\ast \left( \chE_{T^\ast_{\tau \neq 0} (\C^\times \times \C), \, \hat{t},  \hat{\partial}_r}(0) \right)(V) \\
g(r; \tau) &\mapsto g(r;  \tau/r)
\end{split}
\end{align}
where
\begin{equation*}
\chE_{T^\ast_{\tau \neq 0} (\C^\times \times \C), \, \hat{t}, \hat{\partial}_r}(0)= \lbrace P \in \chE_{T^\ast_{\tau \neq 0} \C^\times \times \C}(0) \vert\, [P,r]=0 \textnormal{ and } [P,\partial_t]=0\rbrace.
\end{equation*}

\begin{remark}
	Here, we have used the fact that there are global coordinate systems on $T^\ast \C$ and $T^\ast (\C^\times \times \C)$ which allows us to identify formal microdifferential operators with their total symbols.
\end{remark}

Applying $ (\cdot) \boxtimes \chE_{X}(0)$ to the morphism \eqref{mor:preaction}, we obtain the isomorphism

\begin{equation*}
\gamma^\sharp \boxtimes \id  \colon   \chE_{T^\ast_{\tau \neq 0} (\C^\times \times \C), \, \hat{t}, \hat{\partial}_r}(0)  \boxtimes  \chE_{X}(0) \stackrel{\sim}{\to}  (\gamma \times \id)_\ast \chE_{T^\ast_{\tau \neq 0} (\C^\times \times \C), \, \hat{t}, \hat{\partial}_r}(0)  \boxtimes \chE_{X}(0).
\end{equation*}
Since $\chE_{T^\ast_{\tau \neq 0} (\C^\times \times \C), \, \hat{t}, \hat{\partial}_r}(0)$ and $\chE_{X}(0)$ are sheaves of nuclear Fr\'echet algebras and $\gamma^\sharp$ is continuous, we get, by completing the tensor products, an isomorphism  
\begin{equation}\label{mor:exemplecomp}
\tilde{\gamma}^\sharp \colon   \chE_{T^\ast_{\tau \neq 0} (\C^\times \times \C)\times X, \, \hat{t}, \hat{\partial}_r}(0) \stackrel{\sim}{\to}  (\gamma \times \id)_\ast \chE_{T^\ast_{\tau \neq 0} (\C^\times \times \C)\times X, \, \hat{t}, \hat{\partial}_r}(0).
\end{equation}
Consider the morphism
\begin{equation*}
\widetilde{\rho} \colon T^\ast_{\tau \neq 0} (\C^\times \times \C)\times X \to \C^\times \times X, \; (r,t,x;\lambda,\tau,\xi) \mapsto (r,x; \xi / \tau).
\end{equation*}

Applying $\widetilde{\rho}_\ast$ to the morphism \eqref{mor:exemplecomp} provides the continuous morphism 
\begin{equation*}
\widetilde{\psi} \colon \cA_{\C^\times \times X} \stackrel{\sim}{\to}  \wtmu_\ast \cA_{\C^\times \times X}.
\end{equation*}
By adjunction, we get the $\mathrm{F}$-action
\begin{equation*}
\widetilde{\theta} \colon \wtmu^{-1}\cA_{\C^\times \times X} \stackrel{\sim}{\to}   \cA_{\C^\times \times X}.
\end{equation*}

Let $V$ be an open subset of $\C^\times$ and $(U,x;u)$   a local symplectic coordinate system of $X$ where $U$ is a conical open subset of $X$. Then

\begin{align*}
\widetilde{\theta}_{V \times U} \colon \wtmu^{-1}\cA_{\C^\times \times X}(V \times U) &\to   \cA_{\C^\times \times X}(V \times U)\\ 
\sum_{i \geq 0} f_i(r,x;u) \hbar^i & \mapsto \sum_{i \geq 0} f_i(r,x;r \cdot u) r^i\hbar^i.
\end{align*}
This implies that
\begin{align*}
\theta_r \colon \mu_r^{-1} \chW_X(0)(U) & \to \chW_X(0)(U)\\
\sum_{i \geq 0} f_i(x;u) \hbar^i & \mapsto \sum_{i \geq 0} f_i(x;r \cdot u) r^i\hbar^i.
\end{align*}
A section $s=\sum_{i \geq 0} f_i(x;u) \hbar^i$ in $\chW_X(0)(U)$ is invariant, if for every $r \in \C^\times$, $\theta_r(s)=s$. That is, 
\begin{equation*}
\textnormal{for every $i \geq 0,$\;}f_i(x,r u)=r^{-i}f_i(x,u).
\end{equation*}
This implies that $s \in \iota(\chE_X(0)(U))$ and in particular that $(\chW_X(0))^{\C^\times} \simeq \chE_X(0)$ and $\chW_X^{\C^\times} \simeq \chE_X$. Hence, applying Theorem \ref{thm:equimicroDQ}, we obtain the proposition

\begin{proposition}\label{prop:exemple}
	Assume that $X$ does not intersect the zero sections of $T^\ast M$ and let $Y=X/\C^\times$. Then the adjoint pair $(p^\ast_{\C^\times},p_\ast^{\C^\times})$ induces a well defined adjunction
	\begin{equation*}
	\xymatrix{
		p^\ast_{\C^\times} \colon \Mod_\coh(\chE_Y) \ar@<.4ex>[r]&  \ar@<.4ex>[l]\Mod_{\mathrm{F},\,\coh}(\chW_X) \colon p_\ast^{\C^\times}.
	}
	\end{equation*}
	These functors are equivalence of categories inverse to each others.
\end{proposition}

\section{The codimension three conjecture for formal micro-differential modules}
In this subsection, we deduce the codimension-three conjecture  for formal microdifferential modules from the codimension-three conjecture for DQ-modules. We first recall the DQ-module version of the codimension three conjecture.

\begin{theorem}[{\cite[Theorem 1.5]{codim}}]\label{thm:codim}
	Let $X$ be a complex manifold endowed with a DQ-algebra $\cA_X$ such that the associated Poisson structure is symplectic. Let $\Lambda$ be a closed Lagrangian analytic subset of $X$, $Z$ a closed analytic subset of $\Lambda$ such that $\codim_\Lambda Z \geq 3$ and $j: X \setminus Z \to X$ the open embedding. Let $\cM$ be a holonomic $(\cA^{loc}_X|_{X \setminus Z})$-module, whose support is contained in $\Lambda \setminus Z$. Assume that $\cM$ has an $\cA_X|_{X \setminus Z}$-lattice. Then $j_\ast\cM$ is a holonomic module and is the unique holonomic extension of $\cM$ to $X$ whose support is contained in $\Lambda$. 
\end{theorem}

Let $M$ be a complex manifold and $T^\ast M$ its cotangent bundle, we set $\dot{T}^\ast M= T^\ast M \setminus M$, $P^\ast M$ for the projective cotangent bundle and $p \colon \dot{T} \ast M \to P^\ast M$ for the canonical projection.
Let $X$ be an open subset of $T^\ast M$ and $l$ be a non-negative integer, from now on, we set $X_l:=(\C^\times)^l \times X$, similarly $Z_l:=(\C^\times)^l \times Z$ and  $\cA_{X_l} := \cO^\hbar_{(\C^\times)^l} \ubtimes \chW_X(0)$. We will need the following proposition whose proof is similar to the one of \cite[Theorem 1.2.2]{KasKai}.
\begin{proposition} \label{prop:annulationloccoh}
	Let $r$ and $l$ be non-negative integers and $\cM$ be a coherent $\cA^{loc}_{X_l}$-module so that $\fExt^j_{\cA^{loc}_{X_l}}(\cM,\cA^{loc}_{X_l})=0$ for any $j >r$. Then $\Hn^j_{Z_l}(\cM)=0$ for any closed analytic subset $Z$ of $X$ and any $j < \codim Z -r$.
\end{proposition}

\begin{lemma}\label{lem:ext}
	Let $\Lambda$ be a Lagrangian subvariety of $X$, $Z$ be a closed analytic subset of $\Lambda$ such that $\codim_\Lambda Z \geq 2$, $j\colon X \setminus Z \hookrightarrow X$ the inclusion and $\cM$ a holonomic $\chW_X$-module supported by $\Lambda$. Let $(f,f^\sharp) \colon (X_l, \cA_{X_l}) \to (X, \cA_{X})$ be a morphism of $\C$-ringed space such that $f^\sharp$ is flat. Set $V=X\setminus Z$ and assume that $f^{-1}(V)=V_l$. Let $f_V \colon V_l \to V$ be the restriction of $f$ to $f^{-1}(V)$. Then
	\begin{equation*}
	(\id_{(\C^\times)^{l}} \times j)_\ast f_V^\ast \cM \vert_V \simeq f^\ast \cM.
	\end{equation*}
\end{lemma}

\begin{proof}
	Since by definition of $f_V$ the below diagram commutes
	\begin{equation*}
	\xymatrix{X_l \ar[r]^-{f} & X\\
		V_l \ar[r]^-{f_V} \ar@{_(->}[u]^-{\id_{(\C^\times)^l}\times j} & V \ar@{_(->}[u]_-{j}
	}
	\end{equation*}
	it follows that
	\begin{equation*}
	f_V^\ast \cM \vert_V \simeq  (\id_{(\C^\times)^l} \times j)^{-1} f^\ast \cM.
	\end{equation*}
	Consider the following exact sequence
	\begin{equation*}
	0 \to \Hn^0_{Z_l}(f^\ast \cM) \to f^\ast \cM \to (\id_{(\C^\times)^l} \times j)_\ast (\id_{(\C^\times)^l} \times j)^{-1} f^\ast \cM \to \Hn^1_{Z_l}(f^\ast \cM) \to 0.
	\end{equation*}
	Since $\cM$ is holonomic and $f^\ast$ is exact, it follows that for any $j >d_X/2$ where $d_X$ is the dimension of $X$
	\begin{equation*}
	\fExt^j_{\cA^{loc}_{X_l}}(f^\ast \cM,\cA^{loc}_{X_l})=0.
	\end{equation*}
	Thus, by Proposition \ref{prop:annulationloccoh}, $\Hn^j_{Z_l}(f^\ast \cM)=0$ for $0 \leq j < 2$. Then, the above exact sequence implies that
	\begin{equation*}
	(\id_{(\C^\times)^l} \times j)_\ast (\id_{(\C^\times)^l} \times j)^{-1} f^\ast \cM \simeq f^\ast \cM.
	\end{equation*}
	Hence,
	\begin{equation*}
	(\id_{(\C^\times)^l} \times j)_\ast f_V^\ast \cM \vert_V \simeq f^\ast \cM.
	\end{equation*}
\end{proof}

\begin{lemma}\label{lem:actionextension}
	Assume that $X$ is conical and does not intersect the zero section. Let $\Lambda$ be a conical Lagrangian subvariety of $X$, let $Z$ be a closed conical analytic subset of $\Lambda$ such that $\codim_\Lambda Z \geq 2$, $j\colon X \setminus Z \hookrightarrow X$ the inclusion and $\cM$ a holonomic $\chW_X$- module supported in $\Lambda$ such that $\cM \in \Mod_\mathrm{F}(\chW_X|_{X \setminus Z})$. Then $\cM \in \Mod_\mathrm{F}(\chW_X)$.
\end{lemma}

\begin{proof}
	Set $V=X \setminus Z$; On $\C^\times \times V$, the $\mathrm{F}$-action is given by
	\begin{equation*}
	\phi^\prime \colon \mu_V^\ast \cM \vert_V \to p_{2,V}^\ast\cM \vert_V.
	\end{equation*}
	Applying $(\id_ {\C^\times} \times j)_\ast$, we get
	\begin{equation*}
	(\id_ {\C^\times} \times j)_\ast \phi \colon (\id_ {\C^\times} \times j)_\ast \mu_V^\ast \cM \vert_V \to (\id_ {\C^\times} \times j)_\ast p_{2,V}^\ast \cM \vert_V.
	\end{equation*}
	By Lemma \ref{lem:ext},
	\begin{align*}
	(\id_{\C^\times} \times j)_\ast \mu_V^\ast \cM \vert_V \simeq \mu^\ast \cM, && (\id_{\C^\times} \times j)_\ast p_{2,V}^\ast \cM \vert_V \simeq p_2^\ast \cM.
	\end{align*}
	
	Thus, we obtain a morphism
	\begin{equation*}
	\phi \colon \mu^\ast \cM \to p_2^\ast \cM.
	\end{equation*}
	Applying $(\id_{(\C^\times)^2} \times j)_\ast$ to the below diagram diagram
	\begin{equation*}
	\scalebox{0.78}{\xymatrix{
			(\id_{\C^\times} \times \mu_V)^{\ast} \mu_V^{\ast} \cM\vert_V \ar[rr]^-{(\id_{\C^\times} \times \mu_V)^{\ast}\phi^\prime} \ar@{=}[dd] 
			&& (\id_{\C^\times} \times \mu_V)^{\ast} p_{2,V}^\ast \cM  \vert_V \ar[r]^-{\sim} 
			& q^\ast_{12,V} \mu_V^\ast \cM \vert_V \ar[r]^{q_{12,V}^\ast \phi} 
			& q_{12,V}^\ast p_{2,V}^\ast \cM \vert_V  \ar[d]^-{\wr}\\ 
			&&&&  q_{3,V}^\ast \cM \vert_V \\
			(\sigma \times \id_V)^{\ast} \mu_V^{\ast} \cM\vert_V \ar[rrrr]^-{(\sigma \times \id_V)^{\ast} \phi^\prime} 
			&&&& (\sigma \times \id_V)^{\ast} p_{2,V}^\ast \cM \vert_V \ar[u]_-{\wr}.
	}}
	\end{equation*}
	and using the isomorphisms provided by Lemma \ref{lem:ext}, we obtain a commutative diagram identical to Diagram \eqref{diag:asso}. This shows that $\phi$ is an $\mathrm{F}$-action.
	
\end{proof}

We now deduce the codimension-three conjecture for formal microdifferential modules from the one for DQ-modules.
\begin{theorem}[{\cite[Theorem 1.2]{KV}}]
	Let $M$ be a complex manifold, $X$ an open subset of $T^\ast M$, $\Lambda$ a closed Lagrangian analytic subset of $X$, and $Z$ a closed analytic subset of $\Lambda$ such that $\codim_\Lambda Z \geq 3$. Let $\chE_X$ the sheaf of formal microdifferential operators on $X$ and $\cN$ be a holonomic $(\chE_X|_{X \setminus Z})$-module whose support is contained in $\Lambda \setminus Z$. Assume that $\cN$ possesses an $(\chE_X(0)|_{X \setminus Z})$-lattice $\cN_0$. Then $\cN$ extends uniquely to a holonomic module defined on $X$ whose support is contained in $\Lambda$.
\end{theorem}

\begin{proof}
	\noindent (i) The unicity is proved as in \cite{KV}. Thus we do not repeat the proof here.

	\noindent (ii) By the dummy variable trick, we can assume that $X$ does not intersect the zero section of $T^\ast M$. Applying again the dummy variable trick, we see that it is sufficient to prove the codimension-three conjecture in the following situation: Let $X$ be an open subset of $P^\ast M$, $\chE_X$ the sheaf of microdifferential operators seen as a sheaf on $P^\ast M$ (as it is constant along the fibers of the projection $p \colon \dot{T^\ast}M \to P^\ast M$ this is always possible), $\Lambda$ a closed Lagrangian analytic subset of $X$, and $Z$ a closed analytic subset of $\Lambda$ such that $\codim_\Lambda Z \geq 3$, $\cN$ be a holonomic $(\chE_X|_{X \setminus Z})$-module whose support is contained in $\Lambda \setminus Z$ and possessing an $(\chE_X(0)|_{X \setminus Z})$-lattice $\cN_0$. We now assume that we are in this setting and set $\widetilde{X}=p^{-1}(X)$, $\widetilde{\Lambda}=p^{-1}(\Lambda)$, $\widetilde{Z}=p^{-1}(Z)$. We still write $p \colon \widetilde{X} \to X$ for the restriction of the projection $p \colon \dot{T}^\ast M  \to P^\ast M$. 
	
	We will now apply the codimension-three for DQ-modules. The module $\cM:= p_{\C^\times}^\ast \cN$ belongs to $\Mod_{\mathrm{F}, \coh} (\chW_{\widetilde{X}}|_{\widetilde{X} \setminus \widetilde{Z}})$, has a $\chW_{\widetilde{X}}(0)\vert_{\widetilde{X} \setminus \widetilde{Z}}$-lattice and
	its support is contained in $\widetilde{\Lambda} \setminus \widetilde{Z}$. By \cite[Theorem 1.4]{codim}, $j_\ast \cM$ is an holonomic module supported by $\widetilde{\Lambda}$ and is also endowed with an $\mathrm{F}$-action by Lemma \ref{lem:actionextension}. Moreover it follows from Proposition \ref{prop:exemple} that $p_\ast^{\C^\times}j_\ast\cM$ is a coherent  $\chE_{X}$-module such that
	\begin{equation*}
	(p_\ast^{\C^\times}j_\ast\cM)|_{X \setminus Z} \simeq p_\ast^{\C^\times}((j_\ast\cM)|_{\widetilde{X} \setminus \widetilde{Z}})\simeq \cN.
	\end{equation*}
	This proves that $p_\ast^{\C^\times}j_\ast\cM$ is a coherent extension of $\cN$. This implies that $j_\ast \cN$ is a coherent $\chE_X$-module.
\end{proof}

\end{document}